\documentclass[11pt,letterpaper]{article}
 \usepackage{hyperref}  
\usepackage{color}
\usepackage{bm}
\usepackage{eqnarray}
\usepackage{amssymb,amsmath,amsthm}
\usepackage{subcaption}
\usepackage{multirow}
\usepackage{graphicx}
\usepackage{booktabs}
\usepackage[euler]{textgreek}
\usepackage[numbers,square]{natbib}
\newtheorem{thm}{Theorem}
\newtheorem*{thm*}{Theorem}
\newtheorem{lem}[thm]{Lemma}

\newtheorem{clm}[thm]{Claim}
\newcommand{\llll}[1] {\left #1}
\newcommand{\rrrr}[1] {\right #1}

\newcommand{\dddd}[2]{\dfrac{#1}{#2}}

\newcommand{\aaaa}{\alpha}
\newcommand{\tttt}{\tau}
\newcommand{\ssss}{\sigma}

\newcommand{\dddddd}{\delta}
\newcommand{\llllll}{\lambda}
\newcommand{\bbbb}{\beta}
\newcommand{\GGGG}{\Gamma}
\newcommand{\gggg}{\gamma}
\newcommand{\oooo}{\omega}
\newcommand{\zzzz}{\zeta}

\title{{\bf \normalsize APPROXIMATIONS FOR THE CAPUTO DERIVATIVE (II) }}
\author{Yuri Dimitrov$^1$, Venelin Todorov$^2$, Radan Miryanov$^3$}

\date{}
 
\begin{document}
 \maketitle
 \begin{abstract}
 In the present paper we use the expansion formula of the polylogarithm function to construct approximations of the Caputo derivative which are  related to the midpoint approximation of the integral in the definition of the Caputo derivative.  The asymptotic expansion formula of the Riemann sum approximation of the beta function and the first terms of the expansion formulas of the approximations of the Caputo derivative of the power function are obtained in the paper.
 The induced shifted approximations of the Gr\"unwald formula and the approximations of the Caputo derivative studied in the first part of the paper are constructed and applied for numerical solution of fractional differential equations.\\
{\it MSC 2010\/}:  65D30, 26A33, 34A08, 42A38, 65D32.\\
{\it Key Words and Phrases}: fractional derivative, approximation, Fourier transform, asymptotic expansion formula,  fractional differential equation.
 \end{abstract}

\section{Introduction}
The finite difference schemes for numerical solution of ordinary and partial fractional differential equations (FDEs) involve approximations of the fractional derivatives. The fractional integral of order $\aaaa>0$ and the Caputo fractional derivative  of order $\aaaa$, where $0<\aaaa<1$   are defined as
\begin{align}\label{CD}
&I^\aaaa y(t)=\dddd{1}{\GGGG(\aaaa)}\int_0^t (t-x)^{\aaaa-1} y(x)d x,\nonumber\\
&y^{(\aaaa)}(t)=D^{\aaaa} y(t)=\dddd{1}{\Gamma (1-\aaaa)}\int_0^t \dfrac{y'(x)}{(t-x)^{\aaaa}}d x.
\end{align}
		The  Caputo derivative of the exponential function is expressed with the Mittag-Leffler function as  $D^\aaaa e^t=t^{1-\aaaa}E_{1,2-\aaaa}(t)$. The Caputo derivatives of the sine and cosine functions satisfy
$$ D^\aaaa \sin t=t^{1-\aaaa}E_{2,2-\aaaa}\llll(-t^2\rrrr),\; D^\aaaa \cos t=-t^{2-\aaaa}E_{2,3-\aaaa}\llll(-t^2\rrrr).
$$
 Let $h=t/n$, where $n$ is a positive integer, and $t_m=m h$, $y_m=y(t_m)$.  The L1 approximation for the Caputo derivative is a  commonly used approximation for numerical solution of fractional differential equations \cite{Dimitrov2016,JinLazarovZhou2016,Ma2014}.
\begin{equation}\label{2_1}
  \dfrac{1}{\GGGG(2-\aaaa)h^\alpha}\sum_{k=0}^{n} \ssss_k^{(\alpha)} y_{n-k}=y^{(\alpha)}_n+O\llll(h^{2-\aaaa}\rrrr),
\end{equation}
where $\ssss_0^{(\alpha)}=1$,  $\ssss_n^{(\alpha)}=(n-1)^{1-\aaaa}-n^{1-\aaaa}$ and
$$\ssss_k^{(\alpha)}=(k+1)^{1-\alpha}-2k^{1-\alpha}+(k-1)^{1-\alpha}, \quad (k=2,\cdots,n-1).$$
 The weights $\ssss_k^{(\alpha)}$ of the L1 approximation have the following  properties 
\begin{align}\label{2_2}
&(i)\quad\ssss_0^{(\alpha)}>0,\; \ssss_1^{(\alpha)}<\ssss_2^{(\alpha)}<\cdots<\ssss_k^{(\alpha)}<\cdots<\ssss_{n-1}^{(\alpha)}<0,\;\ssss_{n}^{(\alpha)}<0,\nonumber\\
&(ii)\quad\sum_{k= 0}^n \ssss_k^{(\alpha)} = 0,\quad \sum_{k= 1}^n k \ssss_k^{(\alpha)} = -n^{1-\aaaa},\\
&(iii)\quad\ssss_k^{(\alpha)}= \dddd{C_1}{k^{1+\aaaa}}+O\llll(\dddd{1}{k^{2+\aaaa}} \rrrr),\quad \ssss_n^{(\alpha)}= \dddd{C_2}{n^\aaaa}+O\llll(\dddd{1}{n^{1+\aaaa}} \rrrr),\nonumber
\end{align}
where $C_1=\aaaa(\aaaa-1)$ and $C_2=\aaaa-1$. When the function $y\in C^2[0,t_n]$, the L1 approximation of the Caputo derivative has an order $2-\alpha$ (\cite{LinXu2007}).
 The L1 approximation is constructed by using a linear interpolation of the integrand function on all subintervals $[t_{m-1},t_m]$.  Approximations of fractional integrals and derivatives is an active research field. Higher order approximations of the fractional derivative which use a Lagrange polynomial or spline interpolation of the integrand function  are studied by Odibat \cite{Odibat2006},  Li, Chen and Ye \cite{LiChenYe2011}, Sousa \cite{Sousa2012}, Chen and Deng \cite{ChenDeng2014},  Gao, Sun and Zhang \cite{GaoSunZhang2014}, Yan, Pal and Ford \cite{YanPalFord2014}, Alikhanov \cite{Alikhanov2015}, Li, Cao and Li \cite{LiCaoLi2016}, Kumar, Pandey and Sharma \cite{KumarPandeySharma2018}. Algorithms for computation of the Caputo derivative in $O(N \ln^2 N)$ time are studied by Jiang et al. \cite{JiangZhangZhangZhang2017}, Yan, Sun and Zhang \cite{YanSunZhang2017}, Ren, Mao and Zhang \cite{RenMaoZhang2018}. The two-term equation is a basic ordinary FDE. When $0<\aaaa<1$ the two-term FDE is called fractional relaxation equation.
\begin{equation}\label{TwoTerm1}
y^{(\aaaa)}(t)+y(t)=F(t),\quad y(0)=y_0.
\end{equation} 
Suppose that 
\begin{equation}\tag{*}
\dddd{1}{h^\aaaa}\sum_{k=0}^{n}\llllll_k^{(\aaaa)} y_{n-k}= y_{n}^{(\aaaa)}+O\llll(h^{\bbbb(\aaaa)}\rrrr)
\end{equation}
is an approximation of the Caputo derivative.  The numerical solution of the two term equation \eqref{TwoTerm1} of order $\bbbb(\aaaa)$, which uses  approximation (*) of the Caputo derivative is computed  with $u_0=y_0$ and \cite{Dimitrov2016,Dimitrov2018, DimitrovDimovTodorov2018}
\begin{align}\label{Ns1}\tag{NS1(*)}
 u_n=\dddd{1}{\llllll_0^{(\aaaa)}+h^\aaaa} \llll(h^\aaaa F_n  -\sum_{k=1}^{n}\llllll_k^{(\aaaa)} u_{n-k}\rrrr).
\end{align}

In the first part of the paper \cite{Dimitrov2018} we study the numerical solutions of the  two-term equations
\begin{align}\label{Equation1}
y^{(\aaaa)}(t)+y(t)=1&+t+t^2+t^3+t^4+\dddd{t^{1-\aaaa}}{\GGGG(2-\aaaa)}\\
&+\dddd{2 t^{2-\aaaa}}{\GGGG(3-\aaaa)}
+\dddd{6 t^{3-\aaaa}}{\GGGG(4-\aaaa)}+\dddd{24 t^{4-\aaaa}}{\GGGG(5-\aaaa)},\quad y(0)=1,\nonumber
\end{align}		
\begin{align} \label{Equation2}
y^{(\aaaa)}(t)+y(t)=e^t+t^{1-\aaaa}E_{1,2-\aaaa}(t),\quad y(0)=1,
\end{align}
\begin{equation} \label{Equation3}
y^{(\aaaa)}(t)+y(t)=\cos (2\pi t)-4 \pi^2 t^{1-\aaaa}
E_{2,3-\aaaa}(-4\pi^2 t^2),\quad  y(0)=1.
\end{equation}
Two-term equations \eqref{Equation1}, \eqref{Equation2}, \eqref{Equation3}  have the  solutions 
$y(t)=1+t+t^2+t^3+t^4,\; y(t)=e^t,\; y(t)=\cos (2\pi t).$
 The numerical results for the error and the order of numerical solution NS1\eqref{2_1}, which uses the L1 approximation of the Caputo derivative of two-term equation \eqref{Equation1} and $\aaaa=0.25$, equation \eqref{Equation2},$\aaaa=0.5$ and equation \eqref{Equation3}, $\aaaa=0.75$ are presented in Table 1. The errors of the numerical solutions are computed on the interval $[0,1]$ with respect to the maximum $l_\infty$ norm. 
The Riemann zeta function is a special function defined as
		\begin{equation*}
	\zzzz(\aaaa)=
	\displaystyle{\sum_{n=1}^\infty \dddd{1}{n^\aaaa}}, \;(\aaaa>1),
	\quad	\zzzz(\aaaa)=\displaystyle{\dfrac{1}{1-2^{1-\aaaa}}\sum_{n=1}^\infty \dddd{(-1)^{n-1}}{n^\aaaa}}, \;  (\aaaa>0).  
	\end{equation*}
The Euler-Mclaurin formula 
\begin{align}
h\llll(\dddd{y(0)}{2}+\sum_{k=1}^{n-1}y(k h)+ \dddd{y(t)}{2} \rrrr)=&\int_0^t y(\xi)d\xi+\\
\sum_{n=1}^\infty 2(-1)^{n+1}\zzzz(2n)&\llll(y^{(2n-1)}(t)- y^{(2n-1)}(0) \rrrr)\llll(\dddd{h}{2\pi}\rrrr)^{2n}\nonumber
\end{align}
and the  expansion formulas of the approximations of the fractional integrals and derivatives involve the values of the zeta function.
In \cite{Dimitrov2016,Dimitrov2017} we use the formula for sum of powers
\begin{equation*}
\sum_{k=1}^{n-1}k^\aaaa=\zzzz(-\aaaa)+\dddd{n^{1+\aaaa}}{1+\aaaa}
\sum_{m=0}^{\infty}\binom{\aaaa+1}{m}\dddd{B_m}{n^m}.
\end{equation*}
 to derive the  expansion formula of the L1 approximation 
\begin{align*}
\dddd{1}{\GGGG(2-\aaaa)h^\aaaa}\sum_{k=0}^n \ssss_k^{(\aaaa)} y(t-k h)=y^{(\aaaa)}(t)+&\dddd{\zzzz(\aaaa-1)}{\GGGG(2-\aaaa)}y''(t)h^{2-\aaaa}+O\llll(h^2\rrrr),
\end{align*}
and a second-order approximation for the Caputo derivative by modifying the first three weights of the L1 approximation with the value of the zeta function $\zzzz (\aaaa-1)$. 
\begin{equation}\label{2_4}
y^{(\aaaa)}_n=\dddd{1}{\GGGG(2-\aaaa)h^\aaaa}\sum_{k=0}^n \dddddd_k^{(\aaaa)} y_{n-k}+O\llll(h^{2}\rrrr),
\end{equation}
	\noindent
where $\dddddd_k^{(\aaaa)}=\ssss_k^{(\aaaa)}$ for $2\leq k\leq n$ and
$$\dddddd_0^{(\aaaa)}=\ssss_0^{(\aaaa)}-\zzzz(\aaaa-1),\; \dddddd_1^{(\aaaa)}=\ssss_1^{(\aaaa)}+2\zzzz(\aaaa-1),\; \dddddd_2^{(\aaaa)}=\ssss_2^{(\aaaa)}-\zzzz(\aaaa-1).$$

The Riemann zeta function is a special case $(t=1)$ of the  polylogarithm function defined  as 
$$Li_\aaaa(t)=\sum_{n=1}^{\infty}\dddd{t^n}{n^\aaaa}=t+\dddd{t^2}{2^\aaaa}+\cdots+\dddd{t^n}{n^\aaaa}+\cdots\qquad (|t|<1).$$
The polylogarithm function has properties 
\begin{align}
&Li_\aaaa(t)+Li_\aaaa(-t)=2^{1-\aaaa}Li_\aaaa(t^2),\label{3_1}\\
&Li_\aaaa(t)=\GGGG(1-\aaaa)\llll(\ln \dddd{1}{t}\rrrr)^{\aaaa-1}+\sum_{n=0}^{\infty}\dddd{\zzzz(\aaaa-n)}{n!}\llll(\ln t\rrrr)^n,\label{3_2}
\end{align}
where $\aaaa \neq 1,2,3,\cdots$ and $|\ln t|<2\pi$. From \eqref{3_2} with $t=e^{i \oooo h}$ we obtain
\begin{align}\label{4_1}
Li_\aaaa\llll(e^{i w h}\rrrr)=\GGGG(1-&\aaaa)(-i w)^{\aaaa-1}h^{\aaaa-1}+\zzzz(\aaaa)-(-i w)\zzzz(\aaaa-1)h\\
&+(-i w)^2\dddd{\zzzz(\aaaa-2)}{2}h^2-(-i w)^3\dddd{\zzzz(\aaaa-3)}{6}h^3+O\llll(h^4\rrrr).\nonumber
\end{align}
In   \cite{Dimitrov2018} we use \eqref{4_1} to obtain approximations of the Caputo derivative and their expansion formulas:
\begin{align}\label{4_2}
\dddd{1}{2\GGGG(1-\aaaa)h^\aaaa}\sum_{k=0}^{n}\ssss_k^{(\aaaa)} y_{n-k}=y_n^{(\aaaa)}+O\llll( h^{2-\aaaa} \rrrr),
\end{align}
where $\ssss_0^{(\aaaa)}=1-2\zzzz(\aaaa),\; \ssss_1^{(\aaaa)}=\dddd{1}{2^\aaaa}+2\zzzz(\aaaa)$ and
$$\ssss_k^{(\aaaa)}=\dddd{1}{(k+1)^\aaaa}-\dddd{1}{(k-1)^\aaaa},\quad (k=2,\cdots,n).$$
Approximation \eqref{4_2} has an order $2-\aaaa$ when the function $y(t)\in C^2[0,t_n]$ and satisfies the condition $y(0)=y'(0)=0$. The approximation is extended to all functions of the class $C^2[0,t_n]$ by applying a modification of the last two weights $\ssss_{n-1}^{(\aaaa)}$ and $\ssss_{n}^{(\aaaa)}$:
$$\ssss_{n-1}^{(\aaaa)}=-\dddd{1}{(n-2)^\aaaa}-2 \left(\sum _{k=1}^{n-1} \frac{1}{k^\aaaa}-\frac{n^{1-\aaaa}}{1-\aaaa}-\zzzz (\aaaa)\right),$$
$$\ssss_{n}^{(\aaaa)}=-\dddd{1}{(n-1)^\aaaa}+2 \left(\sum _{k=1}^{n-1} \frac{1}{k^\aaaa}-\frac{n^{1-\aaaa}}{1-\aaaa}-\zzzz (\aaaa)\right).$$
\begin{equation}\label{4_3}
\dddd{1}{\GGGG(-\aaaa)h^\aaaa}\sum_{k=0}^{n}\ssss_k^{(\aaaa)} y_{n-k}= y_n^{(\aaaa)}+O\llll(h^{2-\aaaa}\rrrr),
\end{equation}
where $\ssss_0^{(\aaaa)}=\zzzz(\aaaa)-\zzzz(1+\aaaa),\; \ssss_1^{(\aaaa)}=1-\zzzz(\aaaa)$ and
$$\ssss_k^{(\aaaa)}=\dddd{1}{k^{1+\aaaa}},\quad (k=2,\cdots,n),$$
and modified weights  $\ssss_{n-1}^{(\aaaa)}$ and $\ssss_{n}^{(\aaaa)}$:
$$\ssss_{n-1}^{(\aaaa)}=\dddd{1}{(n-1)^{1+\aaaa}}-\dddd{n^{1-\aaaa}}{\aaaa(1-\aaaa)}+n\llll(\zzzz(1+\aaaa)-\sum_{k=1}^{n-1}\dddd{1}{k^{1+\aaaa}} \rrrr)-
\llll(\zzzz(\aaaa)-\sum_{k=1}^{n-1}\dddd{1}{k^{\aaaa}} \rrrr),$$
$$\ssss_{n}^{(\aaaa)}=(1-n)\llll(\zzzz(1+\aaaa)-\sum_{k=1}^{n-1}\dddd{1}{k^{1+\aaaa}} \rrrr)+\llll(\zzzz(\aaaa)-\sum_{k=1}^{n-1}\dddd{1}{k^{\aaaa}} \rrrr)+\dddd{n^{1-\aaaa}}{\aaaa(1-\aaaa)}.$$
The exponential Fourier transform of the function $y$ is defined as
$$\mathcal{F}[y(t)](w)=\hat{y}(w)=\int_{-\infty}^{\infty}e^{i w t}y(t)d t.
$$
The Fourier transform has properties $\mathcal{F}[y(t-b)](w)=e^{i w b} \hat{y}(w)$ and
$$\mathcal{F}[D^\aaaa y(t)](w)=(-iw)^\aaaa \hat{y}(w),\quad \mathcal{F}[I^\aaaa y(t)](w)=(-iw)^{-\aaaa} \hat{y}(w).$$
The generating function and the expansion formulas of an approximation are related to the Fourier transform of the approximation.  Approximations for the Caputo derivative, constructed from the properties of the Fourier transform and the generating function of the approximation, are studied by Tadjeran, Meerschaert and Scheffer \cite{TadjeranMeerschaertScheffer2006}, Tian, Zhou and Deng \cite{TianZhouDeng2015}, Ding and Li \cite{DingLi2016,DingLi2017}, Dimitrov \cite{Dimitrov2018}, Ren and Wang \cite{RenWang2017}.
 In section 3 we use Fourier transform and  \eqref{4_1} to construct approximations of the Caputo derivative and their expansion formulas which are related to the midpoint approximation of the fractional integral \eqref{CD} in the definition of the Caputo derivtive.
\begin{align}\label{5_1}
\dddd{1}{\GGGG(1-\aaaa)h^\aaaa}\sum_{k=0}^{n} \ssss_k^{(\aaaa)} y_{n-k}= y_n^{(\aaaa)}+O\llll( h^{2-\aaaa} \rrrr),
\end{align}
where
\begin{align*}
\ssss_0^{(\aaaa)}=2^\aaaa-\llll(2^\aaaa-1\rrrr)\zzzz(\aaaa),\quad 
\ssss_1^{(\aaaa)}=\llll(\dddd{2}{3}\rrrr)^\aaaa-2^\aaaa+\llll(2^\aaaa-1\rrrr)\zzzz(\aaaa),
\end{align*}
$$\ssss_{k}^{(\aaaa)}=2^\aaaa\llll(\dddd{1}{(2k+1)^\aaaa}-\dddd{1}{(2k-1)^\aaaa}\rrrr),\quad (k=2,\cdots,n),$$
and modified weights  $\ssss_{n-1}^{(\aaaa)}$ and $\ssss_{n}^{(\aaaa)}$:
$$\ssss_{n-1}^{(\aaaa)}=\dddd{2^\aaaa}{(2n-1)^\aaaa}-\dddd{2^\aaaa}{(2n-3)^\aaaa}+\frac{n^{1-\aaaa}}{1-\aaaa}-\zzzz(\aaaa)-2^\aaaa\llll(\sum _{k=1}^{n} \dddd{1}{(2k-1)^\aaaa}-\zzzz (\aaaa)\rrrr),$$
$$ \ssss_{n}^{(\aaaa)}=-\dddd{2^\aaaa}{(2n-1)^\aaaa}+\zzzz(\aaaa)-\frac{n^{1-\aaaa}}{1-\aaaa}+2^\aaaa\llll(\sum _{k=1}^{n} \dddd{1}{(2k-1)^\aaaa}-\zzzz (\aaaa)\rrrr).$$
\begin{align}\label{5_2}
\dddd{1}{\GGGG(1-\aaaa)h^\aaaa}\sum_{k=0}^{n} \dddddd_k^{(\aaaa)} y_{n-k}= y_n^{(\aaaa)}+O\llll( h^{2} \rrrr),
\end{align}
where  $\dddddd_0^{(\aaaa)}=2^\aaaa+\dddd{3}{2}\llll(1-2^\aaaa\rrrr)\zzzz(\aaaa)+\llll( 2^{\aaaa-1}-1\rrrr)\zzzz(\aaaa-1),$
$$\dddddd_1^{(\aaaa)}=2^\aaaa\llll(\dddd{1}{3^\aaaa}-1 \rrrr)-2\llll(1-2^\aaaa\rrrr)\zzzz(\aaaa)-2\llll( 2^{\aaaa-1}-1\rrrr)\zzzz(\aaaa-1),$$
$$\dddddd_2^{(\aaaa)}=2^\aaaa\llll(\dddd{1}{5^\aaaa}-\dddd{1}{3^\aaaa}\rrrr)+\dddd{1}{2}\llll(1-2^\aaaa\rrrr)\zzzz(\aaaa)+\llll( 2^{\aaaa-1}-1\rrrr)\zzzz(\aaaa-1),$$
$$\dddddd_{k}^{(\aaaa)}=\ssss_{k}^{(\aaaa)},\qquad (3\leq k \leq n),$$
where $\ssss_{k}^{(\aaaa)}$ are the weights of approximation \eqref{5_1}. The L1 approximation  \eqref{2_1} and approximations \eqref{4_2}, \eqref{4_3}, \eqref{5_1}  have an order $2-\aaaa$ and their weights satisfy \eqref{2_2}. The properties of the weights are used in the proofs for the convergence of the numerical solutions as well as to determine and improve the accuracy of the approximations when $n$ is small. 
The accuracy of the numerical solutions of order $2-\aaaa$ which use approximations \eqref{2_1},\eqref{4_2},\eqref{4_3} and \eqref{5_1} for the Caputo derivative depends on the truncations error of the approximations and the absolute value of the  coefficient $C(\aaaa)$ of the term $C(\aaaa) y''(t)h^{2-\aaaa}$ in the asymptotic expansion formulas. Approximations \eqref{2_1},\eqref{4_2},\eqref{4_3} and \eqref{5_1} have coefficients 
$$C_2(\aaaa)=\frac{\zzzz (\aaaa-1)}{\Gamma (2-\aaaa)},C_{13}(\aaaa)=\frac{\zzzz (\aaaa)-2\zzzz (\aaaa-1)}{2\Gamma (1-\aaaa)},C_{14}(\aaaa)=\frac{\zzzz (\aaaa)-\zzzz (\aaaa-1)}{2\Gamma (-\aaaa)},$$
$$C_{15}(\aaaa)=\frac{\llll(2-2^{\aaaa}\rrrr)\zzzz (\aaaa-1)-\llll(2^{\aaaa}-1\rrrr)\zzzz (\aaaa)}{2\Gamma (1-\aaaa)}.$$
\begin{figure}[ht]
  \centering
  \caption{Graph of the  absolute values of coefficients $C_2(\aaaa)$, $C_{13}(\aaaa)$, $C_{14}(\aaaa)$ and $C_{15}(\aaaa)$ for $0<\aaaa<1$.} 
  \includegraphics[width=0.6\textwidth]{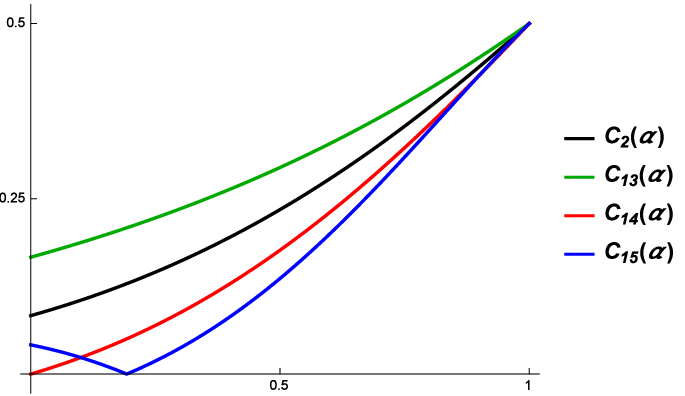}
\end{figure}\\

The absolute values of the coefficients  $C_{14}(\aaaa)$ and $C_{15}(\aaaa)$ are smaller than  the absolute value $|C_1(\aaaa)|$ of the coefficient of the L1 approximation and  $|C_{13}(\aaaa)|>|C_1(\aaaa)|$ (Figure 1).
In the first part of this paper \cite{Dimitrov2018} we show that the errors of numerical solution NS1\eqref{4_3} of two-term  equations \eqref{Equation1},  \eqref{Equation2} and \eqref{Equation3} are  smaller than the errors  of numerical solution NS1\eqref{2_1}, while the corresponding errors of numerical solution NS1\eqref{4_2} are larger. The same pattern is observed when we compare the accuracy of numerical solution NS1\eqref{5_1} to the accuracy of numerical solutions NS1\eqref{2_1} and NS1\eqref{4_2} for standard FDEs.  The errors of numerical solution NS\eqref{5_1} of equations \eqref{Equation1},  \eqref{Equation2} and \eqref{Equation3} are smaller than the errors of numerical solution NS1\eqref{2_1}, which uses the L1 approximation of the Caputo derivative (Table 1 and Table 3).
 In section 4 we construct a finite difference scheme for numerical solution of the fractional subdiffusion equation, which uses approximation \eqref{5_1} of the Caputo derivative and we analyze the convergence of the scheme.
In section 5 we derive the expansion formula of the Riemann sum approximation of the beta function and  the asymptotic expansion formula
\begin{align}\label{AEFB}
\sum_{k=1}^{n-1} k^\bbbb(n-k)^\aaaa=\dddd{\GGGG(\aaaa+1)\GGGG(\bbbb+1)}{\GGGG(\aaaa+\bbbb+2)}&n^{\aaaa+\bbbb+1}+\sum_{k=0}^{\infty} (-1)^k\binom{\bbbb}{k}\zzzz(-\aaaa-k)n^{\bbbb-k}\nonumber\\
+\sum_{k=0}^{\infty}& (-1)^k\binom{\aaaa}{k}\zzzz(-\bbbb-k)n^{\aaaa-k}.
\end{align}
The Gr\"uwald formula approximation of the Caputo derivative has a right endpoint expansion:
\begin{equation*}
\dddd{1}{h^\aaaa}\sum_{k=0}^{N-1} (-1)^k\binom{\aaaa}{k}y(t-k h)=y^{(\aaaa)}(t)+\sum_{k=1}^{M-1}\dddd{B_k^{(-\aaaa)}(-\aaaa)}{k!}y^{(k+\aaaa)}(t)h^k+O\llll( h^M\rrrr),
\end{equation*}
where $B^{(\bbbb)}(t)$ are the generalized Bernoulli polynomials.
When $y(0)=0$ the Gr\"unwald formula approximates the Caputo derivative with a first order accuracy. The Gr\"unwald formula is a second order approximation of the Caputo derivative with a shift parameter $s=S_0(\aaaa)=\aaaa/2$, when the function $y(t)$ satisfies   $y(0)=y'(0)=0$.  Tian, Zhou and Deng \cite{TianZhouDeng2015} construct a class of second order   shifted Gr\"unwald difference (WSGD) approximations. The consruction of the WSGD operators uses a weighted average of shifted Gr\"unwald formula approximations.  Alekhanov \cite{Alikhanov2015} constructs a shifted approximation of the Caputo derivative of order $2-\aaaa$. The approximation is called $\text{L2-1}_s$ formula and has an optimal shift value $s=\aaaa/2$ (in the notations of this paper), where the approximation has an order $3-\aaaa$. In \cite{DimitrovMiryanovTodorov2018} we construct   shifted approximations of order $2-\aaaa$ and two which have a shift parameter $s=\aaaa/2$.  In section 5 and section 6 we use \eqref{AEFB} to show that the first term of the left endpoint expansions of approximations \eqref{2_1}, \eqref{4_2}, \eqref{4_3}, \eqref{5_1} and the shifted Gr\"unwald formula approximation of the power function $t^\bbbb$ is $C h^{1+\bbbb}$, where $C=\zzzz(-\bbbb) t^{-1-\aaaa}/\GGGG(-\aaaa)$.
 In section 6 we extend the method from \cite{DimitrovMiryanovTodorov2018} and we construct the induced shifted approximations of the Caputo derivative of approximations \eqref{2_1}, \eqref{4_2}, \eqref{4_3}, \eqref{5_1} and the Gr\"unwald formula approximation. The induced shifted approximations of  \eqref{2_1}, \eqref{4_2}, \eqref{4_3}, \eqref{5_1} have an order $2-\aaaa$ and a second order accuracy at the optimal shift values. The induced shifted  Gr\"unwald formula approximation has a second order accuracy for an arbitrary shift value and a third order accuracy at the optimal shift value.  In the special case when the shift value is equal to zero the induced Gr\"unwald formula approximation is a second order approximation of the Caputo derivative  for all functions  $y\in C^2[0,t_n]$.
\begin{equation}\label{SOGLA}
\dddd{1}{h^\aaaa}\sum_{k=0}^{N}\gggg_k^{(\aaaa)}y_{n-k}=y_{n}^{(\aaaa)}+O(h^2),
\end{equation}
where 
$$\gggg_k^{(\aaaa)}=(-1)^k \binom{\aaaa}{k}\frac{\aaaa^2 +3 \aaaa-2 k +2}{2 (\aaaa-k+1)},\quad (k=0,\cdots,n-2),$$
$$\gggg_{n-1}^{(\aaaa)}=(-1)^{n-1}\binom{\aaaa-1}{n-2}\frac{2 \aaaa^3-\aaaa^2 n +4 \aaaa^2 -5 \aaaa n+6 \aaaa+2 n^2 -4 n}{2 (\aaaa-1) (\aaaa-n+2)}+\dddd{n^{1-\aaaa}}{\GGGG(2-\aaaa)},
$$
$$\gggg_n^{(\aaaa)}=(-1)^n\binom{\aaaa-1}{n-2}\frac{\aaaa^2 +\aaaa-2 n +2}{2 (\aaaa-1)}-\dddd{n^{1-\aaaa}}{\GGGG(2-\aaaa)}.$$ 

\section{Approximation for the Caputo derivative of order 4-\textalpha} 
Approximation \eqref{4_2} is constructed in \cite{Dimitrov2018} and it is related to the trapezoidal approximation for the fractional integral in the definition of the Caputo derivative \eqref{CD}. The absolute value of the coefficient $C_{13}(\aaaa)$ of the term of order ${2-\aaaa}$ in the expansion formula of approximation \eqref{4_2} is greater than the absolute value of the coefficient $C_{2}(\aaaa)$ of the L1 approximation.  In this section we use the Fourier transform and \eqref{4_1} to derive approximations \eqref{5_1} and \eqref{5_2} for the Caputo derivative, and their expansion formulas of order $4-\aaaa$, which are related to the the midpoint approximation of the fractional integral in the definition of the Caputo derivative. The absolute value of the coefficient $C_{15}(\aaaa)$ of the term of order $2-\aaaa$ in the expansion formula of approximation \eqref{5_1} is smaller than $|C_2(\aaaa)|$. 
From the midpoint approximation for the fractional integral in the definition of the Caputo derivative:
$$y^{(\aaaa)}(t)\approx \dddd{h}{\GGGG(1-\aaaa)}\sum_{k=1}^{n}
\dddd{y'\llll(t_{k-1/2} \rrrr)}{(t-t_{k-1/2})^\aaaa}=\dddd{h^{1-\aaaa}}{\GGGG(1-\aaaa)}\sum_{k=1}^{n}\dddd{y'_{k-1/2}}{(n-k+1/2)^\aaaa}.$$
By approximating $y'_{k-1/2}\approx (y_k-y_{k-1})/{h}$ we obtain
$$y^{(\aaaa)}(t)\approx \dddd{2^\aaaa h^{1-\aaaa}}{\GGGG(1-\aaaa)}\sum_{k=1}^{n}\dddd{y'_{k-1/2}}{(2n-2k+1)^\aaaa}\approx 
\dddd{2^\aaaa }{\GGGG(1-\aaaa) h^\aaaa}\sum_{k=1}^{n}\dddd{y_k-y_{k-1}}{(2n-2k+1)^\aaaa}.$$
Substitute $K=n-k$
$$\GGGG(1-\aaaa)\llll(\dddd{h}{2}\rrrr)^\aaaa y^{(\aaaa)}(t)\approx  \sum_{K=0}^{n-1}\dddd{y_{n-K}-y_{n-K+1}}{(2K+1)^\aaaa}
=   \sum_{k=0}^{n-1}\dddd{y_{n-K}}{(2K+1)^\aaaa}-\sum_{k=0}^{n-1}\dddd{y_{n-K+1}}{(2K+1)^\aaaa},$$
\begin{equation}\label{F18}
y^{(\aaaa)}(t)\approx \dddd{2^\aaaa}{\GGGG(1-\aaaa) h^\aaaa} \llll(\sum_{k=0}^{n-1}\dddd{y_{n-k}}{(2k+1)^\aaaa}-\sum_{k=1}^{n}\dddd{y_{n-k}}{(2k-1)^\aaaa}\rrrr).
\end{equation}
Now we use Fourier transform and \eqref{4_1} to obtain the asymptotic expansion formula of approximation \eqref{F18} of order $4-\aaaa$. Denote
\begin{align}\label{8_1}
S^{(\aaaa)}_n[y(t)]=\sum_{k=0}^{n-1}\dddd{y_{n-k}}{(2k+1)^\aaaa}-\sum_{k=1}^{n}\dddd{y_{n-k}}{(2k-1)^\aaaa}.
\end{align}
\begin{align}\label{8_2}
S^{(\aaaa)}_n[y(t)]=y_n+\sum_{k=1}^{n-1}\llll(\dddd{1}{(2k+1)^\aaaa}-\dddd{1}{(2k-1)^\aaaa}\rrrr)y_{n-k}-\dddd{y_0}{(2n-1)^\aaaa}.
\end{align}
 By applying  Fourier transform to \eqref{8_1}: 
$$\mathcal{F}[S^{(\aaaa)}_\infty[y(t)]](w)=\llll(\sum_{k=0}^{\infty}\dddd{e^{iwkh}}{(2k+1)^\aaaa}-\sum_{k=1}^{\infty}\dddd{e^{iwkh}}{(2k-1)^\aaaa}  \rrrr)\hat{y}(w).
$$
Denote $W=e^{i w h/2}$. 
\begin{align*}
\mathcal{F}[S^{(\aaaa)}_\infty[y(t)]](w)&=\llll(\sum_{k=0}^{\infty}\dddd{W^{2k}}{(2k+1)^\aaaa}-\sum_{k=1}^{\infty}\dddd{W^{2k}}{(2k-1)^\aaaa}  \rrrr)\hat{y}(w),\\
&=\hat{y}(w)\llll(\dddd{1}{W}-W\rrrr)\sum_{k=1}^{\infty}\dddd{W^{2k-1}}{(2k-1)^\aaaa}.
\end{align*}
From \eqref{3_1}
\begin{align*}
2\sum_{k=1}^{\infty}\dddd{W^{2k-1}}{(2k-1)^\aaaa}&=Li_\aaaa(W)-Li_\aaaa(-W) =2Li_\aaaa(W)-Li_\aaaa(W)-Li_\aaaa(-W)\\
&=2Li_\aaaa(W)-2^{1-\aaaa}Li_\aaaa\llll(W^2\rrrr).
\end{align*}
Hence
$$\mathcal{F}[S^{(\aaaa)}_\infty[y(t)]](w)=\hat{y}(w)\llll(e^{-\frac{i w h}{2}}-e^{\frac{i w h}{2}}\rrrr)\llll(Li_\aaaa(e^{\frac{i w h}{2}})-\dddd{1}{2^\aaaa}Li_\aaaa\llll(e^{i w h}\rrrr)\rrrr). $$
The exponential function satisfies
\begin{align}\label{9_1}
e^{-\frac{i w h}{2}}-e^{\frac{i w h}{2}}=(-i w h)+\dddd{(-i w h)^3}{24} +O\llll(h^5\rrrr).
\end{align}
From \eqref{4_1} the function  $Li_\aaaa\llll(e^{\frac{i w h}{2}}\rrrr)$ has a fourth-order series expansion
\begin{align}\label{9_2}
Li_\aaaa\llll(e^{\frac{i w h}{2}}\rrrr)=\dddd{\GGGG(1-\aaaa)}{2^{\aaaa-1}}&(-i w h)^{\aaaa-1}+\zzzz(\aaaa)-(-i w)\dddd{\zzzz(\aaaa-1)}{2}h\\
+(-i w)^2&\dddd{\zzzz(\aaaa-2)}{8}h^2-(-i w)^3\dddd{\zzzz(\aaaa-3)}{48}h^3+O\llll(h^4\rrrr).\nonumber
\end{align}
From \eqref{4_1} and \eqref{9_2} we obtain
\begin{align}\label{9_3}
Li_\aaaa\llll(e^{\frac{i w h}{2}}\rrrr)&-\dddd{1}{2^\aaaa}Li_\aaaa\llll(e^{i w h}\rrrr)=\dddd{\GGGG(1-\aaaa)}{2^{\aaaa}}(-i w)^{\aaaa-1}h^{\aaaa-1}
+\llll(1-\dddd{1}{2^\aaaa}\rrrr)\zzzz(\aaaa)\nonumber\\
-(-i& w)\llll(\dddd{1}{2}-\dddd{1}{2^\aaaa}\rrrr)\zzzz(\aaaa-1)h
+(-i w)^2\dddd{1}{2}\llll(\dddd{1}{4}-\dddd{1}{2^\aaaa}\rrrr)\zzzz(\aaaa-2)h^2\nonumber\\
-(-i& w)^3\dddd{1}{6}\llll(\dddd{1}{8}-\dddd{1}{2^\aaaa}\rrrr)\zzzz(\aaaa-3)h^3+O\llll(h^4\rrrr).
\end{align}
From \eqref{9_1} and \eqref{9_3}:
\begin{align}\label{F24}
\mathcal{F}[S^{(\aaaa)}_\infty&[y(t)]](w)/\hat{y}(w)=\dddd{\GGGG(1-\aaaa)}{2^{\aaaa}}(-i w h)^{\aaaa}
+(-i w)\llll(1-\dddd{1}{2^\aaaa}\rrrr)\zzzz(\aaaa)h\nonumber\\
&-(-i w)^2\llll(\dddd{1}{2}-\dddd{1}{2^\aaaa}\rrrr)\zzzz(\aaaa-1)h^2+\dddd{\GGGG(1-\aaaa)}{24.2^\aaaa}(-i w h)^{2+\aaaa}\\
&+(-i w)^3\llll(\dddd{1}{24}\llll(1-\dddd{1}{2^\aaaa}\rrrr)+\dddd{1}{2}\llll(\dddd{1}{4}-\dddd{1}{2^\aaaa}\rrrr)\rrrr)\zzzz(\aaaa-2)h^3+O\llll(h^4\rrrr).\nonumber
\end{align}
By applying inverse Fourier transform to \eqref{F24}  we obtain the asymptotic expansion formula of order $4-\aaaa$
\begin{align}\label{9_4}
\llll(\dddd{2}{h} \rrrr)^\aaaa& S_n[y]= \GGGG(1-\aaaa)y_n^{(\aaaa)}+2^\aaaa\llll(1-\dddd{1}{2^\aaaa}\rrrr)\zzzz(\aaaa)y'_n h^{1-\aaaa}-\nonumber\\
&\qquad 2^\aaaa \llll(\dddd{1}{2}-\dddd{1}{2^\aaaa}\rrrr)\zzzz(\aaaa-1)y''_n h^{2-\aaaa}+
+\dddd{\GGGG(1-\aaaa)}{24}\dddd{d^2}{dt^2}y_n^{(\aaaa)} h^2\\
& 2^\aaaa \llll(\dddd{1}{24}\llll(1-\dddd{1}{2^\aaaa}\rrrr)\zzzz(\aaaa)+\dddd{1}{2}\llll(\dddd{1}{4}-\dddd{1}{2^\aaaa}\rrrr)\zzzz(\aaaa-2)\rrrr)y'''_n h^{3-\aaaa}+O\llll(h^{4-\aaaa}\rrrr).\nonumber
\end{align}
Let
\begin{equation}\label{10_2}
\mathcal{A}^{(\aaaa)}_n[y]=\dddd{1}{\GGGG(1-\aaaa)h^\aaaa}\sum_{k=0}^{n}\oooo_k^{(\aaaa)} y_{n-k}=y_n^{(\aaaa)}+O\llll( h^{1-\aaaa} \rrrr).
\end{equation}
where
$$\oooo_0^{(\aaaa)}=2^\aaaa,\;  \oooo_k^{(\aaaa)}=2^\aaaa\llll(\dddd{1}{(2k+1)^\aaaa}-\dddd{1}{(2k-1)^\aaaa}\rrrr),\; \oooo_n^{(\aaaa)}=-\dddd{2^\aaaa}{(2n-1)^\aaaa},$$
for $k=1,\cdots,n-1$.
From \eqref{9_4} and  the properties of the inverse Fourier transform we obtain the asymptotic expansion formula of approximation \eqref{10_2} for the Caputo derivative.
\begin{lem} Let $y(0)=y'(0)=y''(0)=0$. Then
\begin{align}\label{10_1}
&\dddd{1}{\GGGG(1-\aaaa)h^\aaaa}\sum_{k=0}^{n}\oooo_k^{(\aaaa)} y_{n-k}= y_n^{(\aaaa)}+\dddd{1}{\GGGG(1-\aaaa)}\llll(2^\aaaa-1\rrrr)\zzzz(\aaaa)y'_n h^{1-\aaaa}-\nonumber\\
 &\qquad \dddd{2^{\aaaa-1}-1}{\GGGG(1-\aaaa)} \zzzz(\aaaa-1)y''_n h^{2-\aaaa}
+\dddd{1}{24}\dddd{d^2}{dt^2}y_n^{(\aaaa)} h^2+\\
&\dddd{1}{\GGGG(1-\aaaa)}\llll(\dddd{1}{24}\llll(2^\aaaa-1\rrrr)\zzzz(\aaaa)+\dddd{1}{2}\llll(2^{\aaaa-2}-1\rrrr)\zzzz(\aaaa-2)\rrrr)y'''_n h^{3-\aaaa}+O\llll(h^{4-\aaaa}\rrrr).\nonumber
\end{align}
\end{lem}
\noindent
 Approximation \eqref{10_2} has an accuracy $O(h^{1-\aaaa})$ and satisfies $\mathcal{A}^{(\aaaa)}_n[1]=D^\aaaa 1=\sum_{k=0}^n \oooo_k^{(\aaaa)}=0$.
The numerical results for the error and the order of numerical solution NS1\eqref{10_2} of two-term equation \eqref{Equation1} and $\aaaa=0.25$, equation \eqref{Equation2} and $\aaaa=0.5$ and equation \eqref{Equation3} with $\aaaa=0.75$ are presented in Table 2.
	
\section{Higher order approximations of the Caputo derivative}
Approximation \eqref{10_2} of the Caputo derivative has an order $1-\aaaa$ and an expansion formula \eqref{10_1} of order $4-\aaaa$. In this section section we use \eqref{10_1} to obtain approximations  of order $2-\aaaa$ and two. 
\subsection{ Approximation for the Caputo derivative of order
2-\textalpha}

By approximating $y'_n$ in \eqref{10_1} with first order backward difference
$$y'_n=\dddd{y_n-y_{n-1}}{h}-\dddd{h}{2}y''_n+O(h^2)$$
 we obtain the approximation for the Caputo derivative
\begin{align}\label{12_1}
\mathcal{A}_n^{(\aaaa)}[y(t)]=\dddd{1}{\GGGG(1-\aaaa)h^\aaaa}\sum_{k=0}^{n}\ssss_k^{(\aaaa)} y_{n-k}= y_n^{(\aaaa)}+O\llll(h^{2-\aaaa}\rrrr),
\end{align}
where  $\oooo_n^{(\aaaa)}=-2^\aaaa/(2n-1)^\aaaa$ and
\begin{equation}\label{weights1}
{\ssss}_0^{(\aaaa)}=\oooo_0^{(\aaaa)}-\llll(2^\aaaa-1\rrrr)\zzzz(\aaaa),\quad {\ssss}_1^{(\aaaa)}=\oooo_1^{(\aaaa)}+\llll(2^\aaaa-1\rrrr)\zzzz(\aaaa)
\end{equation}
\begin{equation}\label{weights2}
\oooo_k^{(\aaaa)}=2^\aaaa\llll(\dddd{1}{(2k+1)^\aaaa}-\dddd{1}{(2k-1)^\aaaa}\rrrr),\quad (1<k<n).
\end{equation}
 Approximation \eqref{12_1} has an order ${2-\aaaa}$ when the function $y\in C^2[0,t_n]$ and   $y(0)=y'(0)=0$,  
and satisfies $\mathcal{A}_n^{({\aaaa})}[1]=D^\aaaa 1=\sum_{k=0}^{n}{\ssss}_k^{(\aaaa)}=0$. Now we apply a modification of the last two weights of approximation \eqref{12_1} in order to extend it to all functions of the class $C^2[0,t_n]$. Denote
$$\mathcal{W}_n =\left(\zzzz(\aaaa)-\frac{n^{1-\aaaa}}{1-\aaaa}+2^\aaaa\llll(\sum _{k=1}^{n} \dddd{1}{(2k-1)^\aaaa}-\zzzz (\aaaa)\rrrr)\right).$$ 
\begin{clm} Let $y(t)=t$. Then
\begin{equation}\label{Claim2}
\mathcal{A}_n^{(\aaaa)}[y(t)]-y^{(\aaaa)}(t)=\dddd{h^{1-\aaaa}W_n}{\GGGG(1-\aaaa)}.
\end{equation}
\end{clm}
\begin{proof}
$$\GGGG(1-\aaaa)h^\aaaa\mathcal{A}_n^{(\aaaa)}[t]=\sum_{k=0}^{n}{\ssss}_k^{(\aaaa)}(t-k h) =-h\sum_{k=1}^{n}k{\ssss}_k^{(\aaaa)},$$
$$\dddd{\GGGG(1-\aaaa)}{h^{1-\aaaa}}\mathcal{A}_n^{(\aaaa)}[t]=\llll(1-2^\aaaa\rrrr)\zzzz(\aaaa)+2^\aaaa\sum_{k=1}^{n-1}\llll( \dddd{k}{(2k-1)^\aaaa}-  \dddd{k}{(2k+1)^\aaaa} \rrrr) +\dddd{2^\aaaa n}{(2 n-1)^\aaaa},$$
\begin{align*}
\mathcal{A}_n^{(\aaaa)}[t]=\dddd{h^{1-\aaaa}}{\GGGG(1-\aaaa)}\llll(\llll(1-2^\aaaa\rrrr)\zzzz(\aaaa)+2^\aaaa\sum_{k=1}^{n}\dddd{1}{(2 k-1)^\aaaa}\rrrr).
\end{align*}
The Caputo derivative of the function $y(t)=t$ satisfies
\begin{align*}
y^{(\aaaa)}(t)=\dddd{t^{1-\aaaa}}{\GGGG(2-\aaaa)}=\dddd{n^{1-\aaaa}h^{1-\aaaa}}{(1-\aaaa)\GGGG(1-\aaaa)}.
\end{align*}
\end{proof}
The goal of the modification procedure  is to ensure that  the modified approximation satisfies   
\begin{align}\label{P12}
\mathcal{A}_n^{(\aaaa)}[1]=0,\qquad \mathcal{A}_n^{(\aaaa)}[t]=t^{1-\aaaa}/\GGGG(2-\aaaa)
\end{align}
  Properties $(ii)$ in \eqref{2_2} of the L1 approximation are equivalent to \eqref{P12}. An approximation of the Caputo derivative which satisfies \eqref{P12} has an order $2-\aaaa$ for small $n$. Let
\begin{align*}
y(t)=y(t)-y(0)-y'(0)t+y(0)+y'(0)t=z(t)+y(0)+y'(0)t.
\end{align*}
The function $z(t)=y(t)-y(0)-y'(0)t$ satisfies $z(0)=z'(0)=0$ and
$$\mathcal{A}_n^{(\aaaa)}[z(t)]=z^{(\aaaa)}(t)+O\llll( h^{2-\aaaa} \rrrr)=y^{(\aaaa)}(t)-y'(0)\dddd{t^{1-\aaaa}}{\GGGG(2-\aaaa)}+O\llll( h^{2-\aaaa} \rrrr).$$
Hence
$$\mathcal{A}_n^{(\aaaa)}[y(t)]=\mathcal{A}_n^{(\aaaa)}[z(t)]+\mathcal{A}_n^{(\aaaa)}[y(0)+y'(0)t],$$
$$\mathcal{A}_n^{(\aaaa)}[y(t)]=y^{(\aaaa)}(t)-y'(0)\dddd{t^{1-\aaaa}}{\GGGG(2-\aaaa)}+y'(0)\mathcal{A}_n^{(\aaaa)}[t]+O\llll( h^{2-\aaaa} \rrrr).$$
From Claim 3
\begin{align*}
\dddd{1}{\GGGG(1-\aaaa)h^\aaaa}\sum_{k=0}^{n}{\ssss}_k^{(\aaaa)} y_{n-k}=y_n^{(\aaaa)}+\dddd{y'_0 \mathcal{W}_n h^{1-\aaaa}}{\GGGG(1-\aaaa)}+O\llll( h^{2-\aaaa} \rrrr).
\end{align*}
By substituting  $h y'_0=y_1-y_{0}+O(h^2)$ we obtain 
\begin{align}\label{13_2}
\mathcal{A}_n^{(\aaaa)}[y(t)]=\dddd{1}{\GGGG(1-\aaaa)h^\aaaa}\sum_{k=0}^{n} \ssss_k^{(\aaaa)} y_{n-k}= y_n^{(\aaaa)}+O\llll( h^{2-\aaaa} \rrrr).
\end{align}
The weights $\ssss^{(\aaaa)}_k$ of approximation \eqref{13_2} are defined with \eqref{weights1} and \eqref{weights2} for $0\leq k\leq n-2$ and the last two weights $\ssss^{(\aaaa)}_{n-1}$ and $\ssss^{(\aaaa)}_n$ are modified as
$$\ssss^{(\aaaa)}_{n-1}=\dddd{2^\aaaa}{(2n-1)^\aaaa}-\dddd{2^\aaaa}{(2n-3)^\aaaa}-\mathcal{W}_n,\quad
 \ssss^{(\aaaa)}_{n}=-\dddd{2^\aaaa}{(2n-1)^\aaaa}+\mathcal{W}_n.$$
Approximation \eqref{13_2} has an order $2-\aaaa$ for all functions 
$y\in  C^2[0,t_n]$. The numerical results for the error and order of numerical solution NS1\eqref{13_2} of two-term equation \eqref{Equation1} and $\aaaa=0.25$, equation \eqref{Equation2} and $\aaaa=0.5$ and equation \eqref{Equation3} with $\aaaa=0.75$ are presented in Table 3.

	\subsection{Second order approximation of the Caputo derivative } 
	By approximating $y'_n$ and $y''_n$ in \eqref{10_1} with
	$$y'_n=\dddd{1}{h}\llll(\dddd{3}{2}y_n-2y_{n-1}+\dddd{1}{2}y_{n-2}   \rrrr)
	+O\llll(h^2\rrrr),$$
		$$y''_n=\dddd{1}{h^2}\llll(y_n-2y_{n-1}+y_{n-2}   \rrrr)
	+O\llll(h\rrrr),$$
	we obtain the approximation for the Caputo derivative
\begin{align}\label{14_1}
\dddd{1}{\GGGG(1-\aaaa)h^\aaaa}\sum_{k=0}^{n} {\dddddd}_k^{(\aaaa)} y_{n-k}= y_n^{(\aaaa)}+O\llll( h^{2} \rrrr),
\end{align}
where
${\dddddd}_0^{(\aaaa)}=\oooo_0^{(\aaaa)}+\dddd{3}{2}\llll(1-2^\aaaa\rrrr)\zzzz(\aaaa)+\llll( 2^{\aaaa-1}-1\rrrr)\zzzz(\aaaa-1),$
$${\dddddd}_1^{(\aaaa)}=\oooo_1^{(\aaaa)}-2\llll(1-2^\aaaa\rrrr)\zzzz(\aaaa)-2\llll( 2^{\aaaa-1}-1\rrrr)\zzzz(\aaaa-1),$$
$${\dddddd}_2^{(\aaaa)}=\oooo_2^{(\aaaa)}+\dddd{1}{2}\llll(1-2^\aaaa\rrrr)\zzzz(\aaaa)+\llll( 2^{\aaaa-1}-1\rrrr)\zzzz(\aaaa-1),$$
 $${\dddddd}_k^{(\aaaa)}=\oooo_k^{(\aaaa)}=2^\aaaa\llll(\dddd{1}{(2k+1)^\aaaa}-\dddd{1}{(2k-1)^\aaaa}\rrrr),\qquad 3\leq k\leq n-2.$$ 
Approximation \eqref{14_1} has an accuracy $O\llll( h^{2} \rrrr)$ when  $y(0)=y'(0)=0$ and a second order accuracy for all function $y\in C^2[0,t_n]$ when the last two weights are defined as
$$\dddddd^{(\aaaa)}_{n-1}=\dddd{2^\aaaa}{(2n-1)^\aaaa}-\dddd{2^\aaaa}{(2n-3)^\aaaa}-\mathcal{W}_n,\quad
 \dddddd^{(\aaaa)}_{n}=-\dddd{2^\aaaa}{(2n-1)^\aaaa}+\mathcal{W}_n.$$
			\begin{figure}[t]
  \centering
  \caption{{Graph of the exact solution of two-term equation \eqref{Equation3} and numerical solutions NS1\eqref{10_2}(green), NS1\eqref{13_2}(red) and NS1\eqref{14_1}(blue) for $\aaaa=0.6,h=0.1$.}} 
  \includegraphics[width=0.6\textwidth]{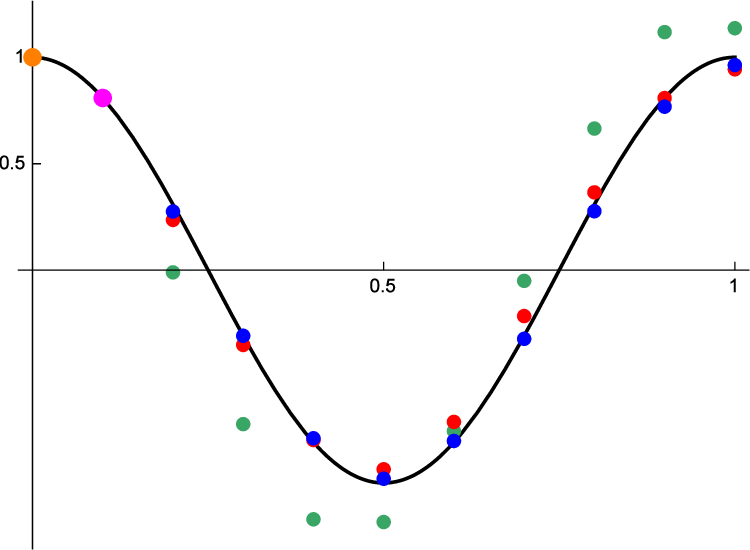}
\end{figure}\\
The order of approximation \eqref{14_1} is $2-\aaaa$ for small $n$, which is sufficient for constructing second order numerical solutions of FDEs.
 The numerical results for the error and the order of second order numerical solution NS1\eqref{14_1}  of two-term equation \eqref{Equation1} and $\aaaa=0.25$, equation \eqref{Equation2}, $\aaaa=0.5$ and equation \eqref{Equation3} with $\aaaa=0.75$ are presented in Table 4.   Second order numerical solution NS1\eqref{14_1} of two-term equation \eqref{Equation3} and $\aaaa=0.6$ is compared to numerical solution NS1\eqref{13_2} of order $1.4$ and NS1\eqref{10_2} of order $0.4$ in Figure 2.
	
\section{Numerical solution of the fractional subdiffusion equation}
The analytical and the numerical solutions of the fractional diffusion equation are studied in \cite{Alikhanov2015, Dimitrov2015,JiSun2015,JiangZhangZhangZhang2017,JinLazarovZhou2016,LiZeng2015,LinXu2007,Ma2014,Podlubny1999,RenMaoZhang2018,SunWu2006, WangVong2014,Wyss1986, YanSunZhang2017, ZengLiLiuTurner2015}. 
 In this section we use approximation \eqref{5_1} of the Caputo derivative  to construct a finite difference scheme for the fractional subdiffusion equation 
	\begin{equation}\label{16_1}
	\left\{
	\begin{array}{l l}
\dfrac{\partial^\alpha u(x,t)}{\partial t^\alpha}=D\dfrac{\partial^2 u(x,t)}{\partial x^2}+F(x,t),\quad (x,t)\in [0,1]\times[0,T],& \\
	u(x,0)=u_0(x),\;u(0,t)=u_1(t),\;u(1,t)=u_2(t),&  \\
	\end{array} 
		\right . 
	\end{equation}
where $0<\aaaa<1$. Let $h=1/N,\tttt=T/M$, where $M$ and $N$ are positive integers, and $\mathcal{J}$ be a grid on the rectangle $[0,1]\times[0,T]$:
$$\mathcal{J}=\llll\{(n h,m \tttt) \| 1\leq n\leq N, 1\leq m\leq M\rrrr\}.$$
Denote by $u_n^m=u(n h,m \tttt)$ and $F_n^m=F(n h,m\tttt)$ the values of the functions $u(x,t)$ and $F(x,t)$ on the grid $\mathcal{J}$.
By approximating the Caputo derivative in the time direction with \eqref{5_1} and  the second derivative in the space direction with a second order  central difference   we obtain
$$\dddd{1}{\GGGG(1-\aaaa)\tttt^\aaaa}\sum_{k=0}^n \ssss_k^{(\aaaa)}u_n^{m-k} = D \dddd{u_{n-1}^m-2 u_n^m+u_{n+1}^m}{h^2}+F_n^m+O\llll(\tttt^{2-\aaaa}+h^2\rrrr).$$
Let $\eta=\GGGG(1-\aaaa)D\tttt^\aaaa/h^2.$
The numerical solution $\{U_n^m\}_{n=1}^{N-1}$ of equation \eqref{16_1} on the $m$-th layer of $\mathcal{J}$ satisfies 
$$-\eta U_{n-1}^m+\llll(\ssss_0^{(\aaaa)}+2\eta\rrrr) U_n^m-\eta U_{n+1}^m=-\sum_{k=1}^n \ssss_k^{(\aaaa)}U_n^{m-k}+\tttt^\aaaa\GGGG(1-\aaaa) F_n^m.$$
	Let $\mathcal{K}$ be a tridiagonal matrix of dimension $N-1$ with values $\ssss_0^{(\aaaa)}+2\eta$ on the main diagonal, and $-\eta$ on the diagonals above and below the main diagonal.
The vector $\mathcal{U}^m=\llll(U_1^m,U_2^m,\cdots,U_{N-1}^m\rrrr)^T$ of the numerical solution on the $m$-th layer of the grid $\mathcal{J}$ is a solution of the linear system
\begin{align}\label{16_2}
	\mathcal{K} \mathcal{U}^m=\mathcal{R}_1+\eta \mathcal{R}_2,
	\end{align}
	where $\mathcal{R}_1$ and $\mathcal{R}_2$ are the column vectors of dimension $N-1$
	$$\mathcal{R}_1^T=\llll[-\sum_{k=1}^n \ssss_k^{(\aaaa)}U_n^{m-k}+\GGGG(1-\aaaa)\tttt^\aaaa F(n h,m\tttt) \rrrr]_{n=1}^{N-1},$$
	$$\mathcal{R}_2^T=\llll[u_1(m\tttt),0,\cdots,0,u_2(m\tttt)\rrrr]^T.$$
The second order numerical solution of the fractional subdiffusion equation on the first layer of the grid $\mathcal{J}$ is computed with the approximation \cite{Dimitrov2016}
		\begin{equation}\label{17_05}
y^{(\aaaa)}(\tttt)=\dddd{y(\tttt)-y(0)}{\tttt^\aaaa\GGGG(2-\aaaa)}+O\llll(\tttt^{2-\aaaa}\rrrr).
		\end{equation}
Let $\tilde{\eta}=\GGGG(2-\aaaa)D\tttt^\aaaa/h^2$. The numerical solution on the first layer of the grid $\mathcal{J}$ satisfies the system of equations
	\begin{equation*}
	\left\{
	\begin{array}{l l}
U_0^1=u_1(\tttt),U_N^1=u_2(\tttt),\quad(n=1\cdots,N-1),& \\
-\tilde{\eta} U_{n-1}^1+(1+2\tilde{\eta}) U_n^1-\tilde{\eta} U_{n+1}^1=U_n^0+\GGGG(2-\aaaa)\tttt^\aaaa F_n^1.&  \\
	\end{array} 
		\right . 
	\end{equation*}
	The fractional subdiffusion equations
		\begin{equation}\label{17_1}
	\left\{
	\begin{array}{l l}
\dfrac{\partial^\alpha v(x,t)}{\partial t^\alpha}=\dfrac{\partial^2 v(x,t)}{\partial x^2}+e^x\llll(t^{1-\aaaa}E_{1,2-\aaaa}(t)-e^t  \rrrr),& \\
	u(x,0)=e^x,\;u(0,t)=e^t, u(1,t)=e^{t+1},\quad (x,t)\in[0,1]\times[0,1],&  \\
	\end{array} 
		\right . 
	\end{equation}
and 
		\begin{equation}\label{17_2}
	\left\{
	\begin{array}{l l}
\dfrac{\partial^\alpha u(x,t)}{\partial t^\alpha}=\dddd{1}{\pi^2}\dfrac{\partial^2 u(x,t)}{\partial x^2},\quad (x,t)\in[0,1]\times[0,1],& \\
	u(x,0)=\sin (\pi x),\;u(0,t)=0, u(1,t)=0,&  \\
	\end{array} 
		\right . 
	\end{equation}
	have the solutions $u(x,t)=e^{x+t}$ and $u(x,t)=\sin (\pi x) E_{\aaaa}\llll(-t^\aaaa\rrrr)$. The numerical results for the error and order of numerical solution \eqref{16_2} of the fractional subdiffusion equations \eqref{17_1} and \eqref{17_2} for $\aaaa=0.25,\aaaa=0.5,\aaaa=0.75$ are presented in Table 5 and Table 6.
	The solution of equation \eqref{17_2} has a singularity at $t=0$,  where its  first partial derivatives are unbounded. 
Numerical solution \eqref{16_2} of equation \eqref{17_2} has a first order accuracy (Table 6). A numerical analysis of the finite difference scheme which uses the L1 approximation for the Caputo derivative is given by Jin, Lazarov and Zhou in \cite{JinLazarovZhou2016}. The Miller-Ross sequential derivative for the Caputo fractional derivative of order $n\aaaa$ is defined as 
	$$y^{[\aaaa]}(t)=D^\aaaa y(t)=y^{(\aaaa)}(t),\quad y^{[n\aaaa]}(t)=D^\aaaa y^{[(n-1)\aaaa]}(t).
	$$
	When the function $y(t)$ is differentiable in the sense of the definition of a Miller-Ross derivative its fractional Taylor polynomials (polyfractonomials) at the initial point of fractional differentiation $t=0$ are defined as:
	$$T_m^{(\aaaa)}(t)= \sum_{k=0}^m\dddd{y^{[k\aaaa]}(0) t^{\aaaa k}}{\GGGG(\aaaa k+1)}.$$
Now we compute the Taylor polyfractonomials of the solution  of the fractional subdiffusion equation and we apply the method from \cite{DimitrovDimovTodorov2018} for transforming equation \eqref{17_2} into a fractional diffusion equation whose solution belongs to the class $C^{2,2}([0,1]\times[0,1])$. Denote by $D_t^{[\bbbb]} u(x,t)$  the  Miller-Ross derivative of order $\bbbb$ of the solution $u(x,t)$ in the time direction. From equation \eqref{17_2}
	$$D_t^{[\aaaa]} u(x,t)=\dddd{1}{\pi^2}\dfrac{\partial^2 u(x,t)}{\partial x^2}.$$
	By applying  fractional differentiation of order $\aaaa$ we obtain
	$$D_t^{[2\aaaa]}u(x,t)=D_t^{\aaaa}D_t^{\aaaa}u(x,t)=\dddd{1}{\pi^2}
	D_t^\aaaa \dfrac{\partial^2 u(x,t)}{\partial x^2} =\dddd{1}{\pi^4}\dfrac{\partial^4 u(x,t)}{\partial x^4},$$
	$$D_t^{[3\aaaa]}u(x,t)=D_t^{\aaaa}D_t^{[2\aaaa]}u(x,t)=\dddd{1}{\pi^4}D_t^\aaaa\dfrac{\partial^4 u(x,t)}{\partial x^4}=\dddd{1}{\pi^6}\dfrac{\partial^6 u(x,t)}{\partial x^6}.$$
By induction we obtain
	$$D_t^{[n\aaaa]}u(x,t)=\dddd{1}{\pi^{2n}}\dfrac{\partial^{2n} u(x,t)}{\partial x^{2n}}.$$
Set $t=0$
		$$D_t^{[n\aaaa]}u(x,0)=\dddd{1}{\pi^{2n}}\dfrac{\partial^{2n} u(x,0)}{\partial x^{2n}}=\dddd{1}{\pi^{2n}}
		\dfrac{\partial^{2n} \sin (\pi x)}{\partial x^{2n}}=(-1)^{n}\sin (\pi x).$$
The solution of equation \eqref{17_2} has Taylor polyfractonomials in  time 
		$$T^{(\aaaa)}_m(x,t)=  \sum_{n=0}^m \dddd{t^{n\aaaa}D_t^{[n\aaaa]}u(x,0)}{\GGGG(n\aaaa+1)}=\sin (\pi x) \sum_{n=0}^m (-1)^n \dddd{h^{n\aaaa}}{\GGGG(n\aaaa+1)}.$$
Substitute
		$$v(x,t)=u(x,t)-T_m^{(\aaaa)}(x,t)=u(x,t)-\sin (\pi x) \sum_{n=0}^m (-1)^n \dddd{t^{n\aaaa}}{\GGGG(n\aaaa+1)}.$$
The function $v(x,t)$ satisfies
\begin{align*}
		&D_t^\aaaa v(x,t)=D_t^\aaaa u(x,t)+\sin (\pi x) \sum_{n=0}^{m-1} (-1)^{n+1} \dddd{t^{n\aaaa}}{\GGGG(n\aaaa+1)},\\
		&\dfrac{\partial^{2}}{\partial x^{2}}v(x,t)=\dfrac{\partial^{2}}{\partial x^{2}}u(x,t)+\pi^2 \sin (\pi x) \sum_{n=0}^{m} (-1)^{n+1} \dddd{t^{n\aaaa}}{\GGGG(n\aaaa+1)},
		\end{align*}
and is a solution of the fractional subdiffusion equation
		\begin{equation}\label{19_1}
	\left\{
	\begin{array}{l l}
	\dfrac{\partial^{\aaaa} v(x,t)}{\partial t^{\aaaa}}=\dddd{1}{\pi^2}\dfrac{\partial^2 v(x,t)}{\partial x^2}+(-1)^{m+1}\sin (\pi x)\dddd{t^{m\aaaa}}{\GGGG(m\aaaa+1)},&  \\
	v(x,0)=v(0,t)=v(\pi,t)=0.&  \\
	\end{array}
		\right . 
	\end{equation}
When $m\aaaa > 2$	the function $v(x,t)$  has a continuous second order partial  derivative in  time   and numerical solution \eqref{16_2} of equation \eqref{19_1} has an accuracy $O\llll(\tau^{2-\aaaa}+h^2\rrrr)$. The numerical results for the error and order of numerical solution \eqref{16_2} of the fractional subdiffusion equation \eqref{19_1} for $\aaaa=0.25$ and $m=8$, $\aaaa=0.5$ and $m=4$ and $\aaaa=0.75,m=2$ are presented in Table 7.
In Theorem 4 we establish the convergence of numerical solution \eqref{16_2} of the fractional subdiffusion equation. The proof of Theorem 4 uses the property of the weights $\ssss_m^{(\aaaa)}$ of approximation \eqref{5_1} from Claim 3.
	\begin{clm}
	$$W_m=\dddd{\aaaa}{24 m^{1+\aaaa}}+O\llll( \dddd{1}{m^{2+\aaaa}}\rrrr).$$
	\end{clm}
	\begin{proof} Let $S_m=\sum_{k=1}^m 1/k^\aaaa$. From the formula for sum of powers
	$$S_{m-1}=\zzzz(\aaaa)+\dddd{m^{1-\aaaa}}{1-\aaaa}\sum_{k=0}^\infty \binom{1-\aaaa}{k}\dddd{B_k}{m^k}.
	$$
	The numbers $S_m$ satisfy
\begin{align*}
	S_m=\zzzz(\aaaa)+\frac{m^{1-\aaaa}}{1-\aaaa}+\dddd{1}{2 m^\aaaa}-\dddd{\aaaa}{12m^{1+\aaaa}}+O\llll( \dddd{1}{m^{2+\aaaa}}\rrrr),
	\end{align*}
	 $$S_{2m}=\sum_{k=1}^{2m} 1/k^\aaaa=\sum _{k=1}^{m} \dddd{1}{(2k-1)^\aaaa}+\sum _{k=1}^{m} \dddd{1}{(2k)^\aaaa}=\sum _{k=1}^{m} \dddd{1}{(2k-1)^\aaaa}+\dddd{S_m}{2^\aaaa},$$
	$$S_{2m}-\dddd{S_m}{2^\aaaa}=\zzzz(\aaaa)\llll( 1-\dddd{1}{2^\aaaa}\rrrr)+\dddd{m^{1-\aaaa}}{(1-\aaaa)2^\aaaa}+\dddd{\aaaa}{12 (2m)^{1+\aaaa}}+O\llll( \dddd{1}{m^{2+\aaaa}}\rrrr).$$
	Hence
	$$W_m=\zzzz(\aaaa)-\frac{m^{1-\aaaa}}{1-\aaaa}+2^\aaaa\llll(S_{2m}-\dddd{S_m}{2^\aaaa}-\zzzz (\aaaa)\rrrr)=\dddd{\aaaa}{24m^{1+\aaaa}}+O\llll( \dddd{1}{m^{2+\aaaa}}\rrrr).$$
	\end{proof}
	\noindent
 From Claim 4:
\begin{align}\label{20_2}
	\llll|\ssss_m^{(\aaaa)}\rrrr|>\dddd{2^\aaaa}{(2m-1)^\aaaa}-W_m>\dddd{1}{m^\aaaa}-\frac{\aaaa}{24 m^{1+\aaaa}}>\dddd{23}{24 m^\aaaa}>\dddd{1}{2m^\aaaa}.
	\end{align}
	 The maximum (infinity) norm  of the vector $\mathcal{V}=\llll(v_i\rrrr)$ and the square matrix  $\mathcal{L}=\llll(l_{ij}\rrrr)$ of dimension $N-1$ are defined as
$$\llll\| \mathcal{V}\rrrr\|=\max_{1\leq i\leq N-1} |v_i|,\quad \llll\| \mathcal{L} \rrrr\|=\max_{1\leq j\leq N-1} \sum_{m=1}^{N-1}|l_{ij}|.$$
The matrix $\mathcal{K}$ is a  diagonally dominant tridiagonal matrix with positive elements on the main diagonal and negative elements on the diagonals below and above the main diagonal. The matrix $\mathcal{K}^{-1}$ is a positive matrix. 
From the  Ahlberg-Nilson-Varah bound \cite{AhlbergNilson1963,Varah1975}
$$\llll\| \mathcal{K}^{-1} \rrrr\|\leq \dddd{1}{\ssss_0^{(\aaaa)}}.$$
The numbers $\ssss_k^{(\aaaa)}$ satisfy
$\sum_{k=1}^m \llll| \ssss_k^{(\aaaa)} \rrrr|=\ssss_0^{(\aaaa)}$. From \eqref{20_2}
\begin{align}\label{21_1}
\sum_{k=1}^{m-1} \llll| \ssss_k^{(\aaaa)} \rrrr|=\ssss_0^{(\aaaa)}-\llll|\ssss_m^{(\aaaa)}\rrrr|<\ssss_0^{(\aaaa)}-\dddd{1}{2m^\aaaa}.
\end{align}
Let $e_n^m=u_n^m-U_n^m$ be the error of numerical solution \eqref{16_2}. The error vector $\mathcal{E}^m=(e_n^m)$ on the $m$-th layer of the grid $\mathcal{J}$ satisfies the system of equations
$$\mathcal{K} \mathcal{E}^m=\mathcal{R}^m,$$
where $\mathcal{R}^m=(r^m_n)$ is an $N-1$ dimensional column vector with elements
$$r_n^m=-\sum_{k=1}^{m-1} \ssss_k^{(\aaaa)}e_n^{m-k}+\tttt^\aaaa\llll(A_n^m \tttt^{2-\aaaa} +B_n^m h^2\rrrr),$$
and $A_n^m \tttt^{2-\aaaa} +B_n^m h^2$ is the truncation error at the point $(n h,m\tttt)$. 
\begin{thm} The error of \eqref{16_2} on the $m$-th layer of   $\mathcal{J}$ satisfies
\begin{align}\label{21_2}
\llll\| \mathcal{E}^m\rrrr\|\leq 2A m^\aaaa \tttt^\aaaa\llll(\tttt^{2-\aaaa}+h^2\rrrr).
\end{align}
\end{thm}
\begin{proof} Induction on $m$. Let $A$ be large enough, such that  $\llll|A_n^m\rrrr|<A$ and $\llll|B_n^m\rrrr|<A$, for all $m,n$ and \eqref{21_2} holds for $m=1$ and $m=2$. Suppose that \eqref{21_2} holds for all $k\leq m-1$:
$$\llll|r_n^m\rrrr|\leq \sum_{k=1}^{m-1} \llll|\ssss_k^{(\aaaa)}\rrrr|\llll|e_n^{m-k}\rrrr|+A \tttt^\aaaa\llll( \tttt^{2-\aaaa} +h^2\rrrr).$$
From the induction assumption and \eqref{21_1}
$$\llll|r_n^m\rrrr|\leq 2A m^\aaaa \tttt^\aaaa\llll(\tttt^{2-\aaaa}+h^2\rrrr)\sum_{k=1}^{m-1} \llll|\ssss_k^{(\aaaa)}\rrrr|+A \tttt^\aaaa \llll(\tttt^{2-\aaaa} +h^2\rrrr),$$
$$\llll|r_n^m\rrrr|\leq 2A m^\aaaa \tttt^\aaaa\llll(\tttt^{2-\aaaa}+h^2\rrrr)\llll(\ssss_0^{(\aaaa)}-\dddd{1}{2 m^\aaaa}\rrrr)+A \tttt^\aaaa\llll( \tttt^{2-\aaaa} +h^2\rrrr),$$
$$\llll|r_n^m\rrrr|\leq 2A\ssss_0^{(\aaaa)} m^\aaaa\tttt^\aaaa\llll(\tttt^{2-\aaaa}+h^2\rrrr),$$
for all $n=1,\cdots,N-1$. Then
$$\llll\| \mathcal{R}^m\rrrr\|\leq 2A\ssss_0^{(\aaaa)} m^\aaaa\tttt^\aaaa\llll(\tttt^{2-\aaaa}+h^2\rrrr).
$$
The error on the $m$-th layer of the grid $\mathcal{J}$ satisfies $\mathcal{E}^m=\mathcal{K}^{-1}\mathcal{R}^m$ and
$$\llll\| \mathcal{E}^m\rrrr\|\leq 
\llll\| \mathcal{K}^{-1}\rrrr\|\llll\| \mathcal{R}^m\rrrr\|\leq \dddd{1}{\ssss_0^{(\aaaa)}}\llll\| \mathcal{R}^m\rrrr\|\leq 2A m^\aaaa\tttt^\aaaa\llll(\tttt^{2-\aaaa}+h^2\rrrr).
$$
\end{proof}
From \eqref{21_2}  the error of numerical solution \eqref{16_2} of the fractional subdiffusion equation on the grid $\mathcal{J}$ satisfies
$$\llll\| \mathcal{E}^m\rrrr\|\leq  2A M^\aaaa\tttt^\aaaa\llll(\tttt^{2-\aaaa}+h^2\rrrr)\leq  2A T^\aaaa\llll(\tttt^{2-\aaaa}+h^2\rrrr),
$$
for all $m=1,2,\cdots,M$.
\section{Approximations for the beta function and the Caputo derivative of the power function}
In  this section we obtain the  expansion formula of the Riemann sum approximation of the beta function. The expansion formula is used  to find the first term of the left endpoint expansions of approximations \eqref{2_1}, \eqref{4_2}, \eqref{4_3} and \eqref{5_1} of the Caputo derivative of the power function.
 The Riemann sum approximation of the fractional integral has an asymptotic expansion formula
\begin{align} \label{AFI}
 h^{\aaaa}&\sum_{k=1}^{n-1} \dddd{y(t-k h)}{k^{1-\aaaa}}=\int_0^t \dddd{y(x)}{(t-x)^{1-\aaaa}}dx +\sum_{k=0}^\infty (-1)^k\dddd{\zzzz(1-\aaaa-k)}{k!}y^{(k)}(t) h^{k+\aaaa}-\nonumber\\
&\GGGG(\aaaa)\sum_{k=0}^\infty \dddd{B_{k+1}}{(k+1)!}\llll( \sum_{m=0}^k (-1)^m\binom{k}{m}\dddd{x^{\aaaa-m-1}}{\GGGG(\aaaa-m)}y^{(m-k)}(0) \rrrr)h^{k+1}.
\end{align}
The right endpoint expansion formula is obtained from the expansion formula \eqref{4_1} of the polylogarithm function $Li_{1-\aaaa}\llll(e^{i w h}\rrrr)$. In the special case $\aaaa=1$ asymptotic expansion formula \eqref{AFI} is the Euler-Mclaurin formula. 
\subsection{Riemann sum approximation of the beta function}  The beta function is defined  as 
\begin{align}\label{BetaF}
B(\aaaa,\bbbb)=\dddd{\GGGG(\aaaa)\GGGG(\bbbb)}{\GGGG(\aaaa+\bbbb)}=\int_0^1 (1-x)^{\aaaa-1} x^{\bbbb-1} dx.
\end{align}
From \eqref{BetaF} with $\aaaa:=\aaaa+1,\bbbb:=\bbbb+1$ and $x:= t x$ we obtain
\begin{align}\label{BetaI}
\int_0^t x^\bbbb (t-x)^\aaaa dx=B(\aaaa+1,\bbbb+1)t^{\aaaa+\bbbb+1},
\end{align}
where $\aaaa>-1,\bbbb>-1$. Let $y(x)=x^\bbbb, z(x)=(t-x)^\aaaa$.
The derivatives of the functions $y(x)$ and $z(x)$ satisfy
$$\dddd{y^{(k)}(t)}{k!}=\binom{\bbbb}{k}t^{\bbbb-k}, \quad \dddd{z^{(k)}(0)}{k!}=(-1)^k\binom{\aaaa}{k}t^{\aaaa-k}.$$
The Riemann sum approximation of \eqref{BetaI}  satisfies
$$h\sum_{k=1}^{n-1} (kh)^\bbbb(t-kh)^\aaaa=\int_0^t x^\bbbb (t-x)^\aaaa d x+L_0(h)+R_t(h),$$
$$h^{\aaaa+\bbbb+1}\sum_{k=1}^{n-1} k^\bbbb(n-k)^\aaaa=B(\aaaa+1,\bbbb+1)t^{\aaaa+\bbbb+1}+L_0(h)+R_t(h).$$
 From \eqref{AFI} with $\bbbb :=\bbbb+1$, the left endpoint expansion $L_0(h)$ satisfies
\begin{align*}
L_0(h)=&\sum_{k=0}^\infty (-1)^k\dddd{\zzzz(-\bbbb-k)}{k!}z^{(k)}(0)h^{k+\bbbb+1}=\\
&\sum_{k=0}^\infty (-1)^k\binom{\aaaa}{k}\zzzz(-\bbbb-k)t^{\aaaa-k}h^{k+\bbbb+1}.
\end{align*}
 From \eqref{AFI} with $\aaaa :=\aaaa+1$, the right endpoint expansion $R_t(h)$ satisfies
\begin{align*}
R_t(h)=&\sum_{k=0}^\infty (-1)^k\dddd{\zzzz(-\aaaa-k)}{k!}y^{(k)}(t)h^{k+\aaaa+1}=\\
&\sum_{k=0}^\infty (-1)^k\binom{\bbbb}{k}\zzzz(-\aaaa-k)t^{\bbbb-k}h^{k+\aaaa+1}.
\end{align*}
Hence
\begin{align}\label{AA1}
h^{\aaaa+\bbbb+1}\sum_{k=1}^{n-1} k^\bbbb(n-k)^\aaaa=&\dddd{\GGGG(\aaaa+1)\GGGG(\bbbb+1)}{\GGGG(\aaaa+\bbbb+2)}t^{\aaaa+\bbbb+1}+\\
&\sum_{k=0}^\infty (-1)^k\binom{\bbbb}{k}\zzzz(-\aaaa-k)t^{\bbbb-k}h^{k+\aaaa+1}+\nonumber\\
&\sum_{k=0}^\infty (-1)^k\binom{\aaaa}{k}\zzzz(-\bbbb-k)t^{\aaaa-k}h^{k+\bbbb+1}.\nonumber
\end{align}
From \eqref{AA1} with $t=t_n=nh$ we obtain the asymptotic expansion formula 
\begin{align}\label{AEFSP}
\sum_{k=1}^{n-1} k^\bbbb (n-k)^\aaaa&=\dddd{\GGGG(\aaaa+1)\GGGG(\bbbb+1)}{\GGGG(\aaaa+\bbbb+2)}n^{\aaaa+\bbbb+1}+\sum_{k=0}^{M-1} (-1)^k\binom{\bbbb}{k}\zzzz(-\aaaa-k)n^{\bbbb-k}+\nonumber\\
\sum_{k=0}^{N-1}& (-1)^k\binom{\aaaa}{k}\zzzz(-\bbbb-k)n^{\aaaa-k}+O\llll( n^{\min\{\bbbb-M,\aaaa-N\} } \rrrr).
\end{align}
The values of the zeta function at the negative integers satisfy
$$\zzzz(-n)=(-1)^n\dddd{B_{n+1}}{n+1}.$$
From \eqref{AEFSP} with $\bbbb=0,\aaaa:=-\aaaa$ we obtain the formula for sum of powers 
\begin{align*}
\sum_{k=1}^{n-1} \dddd{1}{k^\aaaa}=\zzzz(\aaaa)+\dddd{n^{1-\aaaa}}{1-\aaaa}+\sum_{k=0}^{N-1} (-1)^k\binom{-\aaaa}{k}\zzzz(-k)n^{-\aaaa-k}+O\llll( n^{-\aaaa-N} \rrrr),
\end{align*}
\begin{align*}
\sum_{k=1}^{n-1} \dddd{1}{k^\aaaa}=\zzzz(\aaaa)+\dddd{n^{1-\aaaa}}{1-\aaaa}+\sum_{k=0}^{N-1} \binom{-\aaaa}{k}\dddd{B_{k+1}}{k+1}n^{-\aaaa-k}+O\llll(\dddd{1}{ n^{\aaaa+N}} \rrrr),
\end{align*}
\begin{align*}
\sum_{k=1}^{n-1} \dddd{1}{k^\aaaa}=\zzzz(\aaaa)+\dddd{n^{1-\aaaa}}{1-\aaaa}\sum_{m=0}^{N} \binom{1-\aaaa}{m}\dddd{B_{m}}{n^m}+O\llll(\dddd{1}{ n^{\aaaa+N}} \rrrr).
\end{align*}
\subsection{Approximations for Caputo derivative of the power function} The power function $y(t)=t^\bbbb$ has a Caputo derivative
\begin{align}\label{CDPF}
D^\aaaa t^\bbbb=\dddd{\bbbb}{\GGGG (1-\aaaa)}\int_0^t \dddd{x^{\bbbb-1}}{(t-x)^\aaaa}d x=\dddd{\GGGG(1+\bbbb)}{\GGGG(1+\bbbb-\aaaa)}t^{\bbbb-\aaaa},
\end{align}
 where $\bbbb>0$. In this section we show that the first term of the left endpoint expansions of approximations \eqref{2_1}, \eqref{4_2}, \eqref{4_3}, \eqref{5_1}  of the Caputo derivative of the power function  is $Ch^{1+\bbbb}$, where $C=\zzzz(-\bbbb) t^{-1-\aaaa}/\GGGG(-\aaaa)$.
\begin{clm} Let $y(t)=t^\bbbb$. Then
\begin{align}\label{C5}
\dddd{1}{\GGGG(-\aaaa)h^\aaaa}\llll(\sum_{k=1}^{n-1}\dddd{y_{n-k}}{k^{1+\aaaa}}-\zzzz(1+\aaaa)y_n+\zzzz(\aaaa)y'_n h\rrrr)&=\\
y_n^{(\aaaa)}+\dddd{\zzzz(-\bbbb)}{\GGGG(-\aaaa)t^{1+\aaaa}}&h^{1+\bbbb}+O\llll(h^{2-\aaaa}\rrrr).\nonumber
\end{align}
\end{clm}
\begin{proof}
From \eqref{AA1} with $\aaaa:=-1-\aaaa$ we obtain
\begin{align}\label{EBF1}
h^{\bbbb-\aaaa}\sum_{k=1}^{n-1} k^{-1-\aaaa}(n-k)^\bbbb&=\dddd{\GGGG(-\aaaa)\GGGG(1+\bbbb)}{\GGGG(1+\bbbb-\aaaa)}t^{\bbbb-\aaaa}+\\
&\sum_{k=0}^\infty (-1)^k\binom{\bbbb}{k}\zzzz(1+\aaaa-k)t^{\bbbb-k}h^{k-\aaaa}+\nonumber\\
&\sum_{k=0}^\infty (-1)^k\binom{-1-\aaaa}{k}\zzzz(-\bbbb-k)t^{-1-\aaaa-k}h^{k+\bbbb+1}.\nonumber
\end{align}
From \eqref{EBF1} the Riemann sum approximation of  the singular integral in \eqref{CDPF} has an expansion formula of order $2-\aaaa$
\begin{align*}
h^{\bbbb-\aaaa}\sum_{k=1}^{n-1} k^{-1-\aaaa}(n-k)^\bbbb=\GGGG(-\aaaa)D^\aaaa t^\bbbb&+\zzzz(1+\aaaa)t^\bbbb\dddd{1}{h^\aaaa}-\bbbb\zzzz(\aaaa)t^{\bbbb-1}h^{1-\aaaa}+\\
&\dddd{\zzzz(-\bbbb)}{t^{1+\aaaa}}h^{1+\bbbb}+O\llll(h^{2-\aaaa}\rrrr).
\end{align*}
Set $t=t_n=nh$:
\begin{align*}
\dddd{1}{h^{\aaaa}}\sum_{k=1}^{n-1} \dddd{y_{n-k}}{k^{1+\aaaa}}=\GGGG(-\aaaa) y_n^{(\aaaa)}+&\dddd{1}{h^\aaaa}\zzzz(1+\aaaa)y_n-\zzzz(\aaaa)h^{1-\aaaa}y'_n+\\
&\dddd{\zzzz(-\bbbb)}{t_n^{1+\aaaa}}h^{1+\bbbb}+O\llll(h^{2-\aaaa}\rrrr).
\end{align*}
\end{proof}
From \eqref{C5} and the substitution $h y'_n=y_n-y_{n-1}+O(h^2)$ we obtain the expansion formula of order $2-\aaaa$ of approximation \eqref{4_3} of the power function. The order of approximation\eqref{4_3} of $t^\bbbb$ is $\min\{1+\bbbb,2-\aaaa\}$, when $0<\aaaa,\bbbb<1$. In Claim 6, Clam 7 and Claim 8 we show that the left endpoint expansion formulas of  approximations \eqref{2_1}, \eqref{4_2} and \eqref{5_1} of the Caputo derivative of the power function have the same first term of order $1+\bbbb$.
\begin{clm} Let $y(t)=t^\bbbb$. Then
\begin{align}\label{C6}
&\dddd{1}{\GGGG(2-\aaaa)h^\aaaa}\llll(\sum_{k=1}^{n-1}\llll((k+1)^{1-\aaaa}-2k^{1-\aaaa}+(k-1)^{1-\aaaa} \rrrr)y_{n-k}+y_n\rrrr)=\nonumber\\
&\qquad\qquad\qquad\qquad\qquad\qquad y_n^{(\aaaa)}+\dddd{\zzzz(-\bbbb)}{\GGGG(-\aaaa)t^{1+\aaaa}}h^{1+\bbbb}+O\llll(h^{2-\aaaa}\rrrr).
\end{align}
\end{clm}
\begin{proof} Let 
$z(x)=(t-x+h)^{1-\aaaa}-2(t-x)^{1-\aaaa}+(t-x-h)^{1-\aaaa}.$ The L1 approximation of the power function is expressed with $z(x)$ as
\begin{align}\label{Cz6}
&\dddd{1}{\GGGG(2-\aaaa)h^\aaaa}\llll(\sum_{k=1}^{n-1}\llll((k+1)^{1-\aaaa}-2k^{1-\aaaa}+(k-1)^{1-\aaaa} \rrrr)y_{n-k}+y_n\rrrr)=\nonumber\\
&\quad\dddd{1}{\GGGG(2-\aaaa)h}\llll(\sum_{k=1}^{n-1}z_{n-k}y_k +y_n\rrrr)=\dddd{h^{\bbbb-1}}{\GGGG(2-\aaaa)}\llll(\sum_{k=1}^{n-1} z_{n-k}k^\bbbb+y_n\rrrr)
\end{align}
From \eqref{AFI} with $\aaaa=1+\bbbb$ the left endpoint expansion of \eqref{Cz6} satisfies
$$L_0(h)=\dddd{h^{\bbbb-1}}{\GGGG(2-\aaaa)}\sum_{k=0}^\infty \dddd{\zzzz(-\bbbb-k)}{k!}z^{(k)}(0)h^{k}=\dddd{\zzzz(-\bbbb)z(0)}{\GGGG(2-\aaaa)}h^{\bbbb-1}+O\llll(h^{2+\bbbb}\rrrr).
$$
 From second order central difference approximation
$$\dddd{z(0)}{h^2}=\dddd{(t+h)^{1-\aaaa}-2t^{1-\aaaa}+(t-h)^{1-\aaaa}}{h^2}=\dddd{d^2\llll(t^{1-\aaaa}\rrrr)}{dt^2}+O\llll(h^2\rrrr),$$
$$\dddd{z(0)}{\GGGG(2-\aaaa)h^2}=\dddd{(1-\aaaa)(-\aaaa)t^{-1-\aaaa}}{\GGGG(2-\aaaa)}+O\llll(h^2\rrrr)=\dddd{1}{\GGGG(-\aaaa)t^{1+\aaaa}}+O\llll(h^2\rrrr).$$
Therefore the first term of the left endpoint expansion of the L1 approximation of the power function is $\zzzz(-\bbbb) t^{-1-\aaaa}h^{1+\bbbb}/\GGGG(-\aaaa)$.
\end{proof}
From \eqref{C6} with $\bbbb=\aaaa$ the L1 approximation of the Caputo derivative of the power function $t^\aaaa$ has an order $1+\aaaa$ when $0<\aaaa\leq 0.5$, and an order $2-\aaaa$ when $0.5<\aaaa<1$. The  L1 approximation of the Caputo derivative of $t^\aaaa$ has an accuracy $O\llll(h^\aaaa\rrrr)$ when $n$ is small. Approximations \eqref{4_2}, \eqref{4_3} and \eqref{5_1}  of the Caputo derivative of $t^\aaaa$ have the same order.
\begin{clm} Let $y(t)=t^\bbbb$. Then
\begin{align*}
\dddd{1}{2\GGGG(1-\aaaa)h^\aaaa}&\llll(\sum_{k=1}^{n-1}\llll(\dddd{1}{(k+1)^{\aaaa}}-\dddd{1}{(k-1)^{\aaaa}} \rrrr)y_{n-k}+y_n\rrrr)=\\
&y_n^{(\aaaa)}+\dddd{1}{\GGGG(1-\aaaa)}\zzzz(\aaaa)y'_n h^{1-\aaaa}+\dddd{\zzzz(-\bbbb)}{\GGGG(-\aaaa)t^{1+\aaaa}}h^{1+\bbbb}+O\llll(h^{2-\aaaa}\rrrr).
\end{align*}
\end{clm}
\begin{clm} Let $y(t)=t^\bbbb$. Then
\begin{align*}
\dddd{2^\aaaa}{\GGGG(1-\aaaa)h^\aaaa}&\llll(\sum_{k=1}^{n-1}\llll(\dddd{1}{(2k+1)^{\aaaa}}-\dddd{1}{(2k-1)^{\aaaa}} \rrrr)y_{n-k}+y_n\rrrr)=\\
&y_n^{(\aaaa)}+\dddd{2^\aaaa-1}{\GGGG(1-\aaaa)}\zzzz(\aaaa)y'_n h^{1-\aaaa}+\dddd{\zzzz(-\bbbb)}{\GGGG(-\aaaa)t^{1+\aaaa}}h^{1+\bbbb}+O\llll(h^{2-\aaaa}\rrrr).
\end{align*}
\end{clm}
\noindent
The proofs of Claim 8 and Claim 9 are similar to the proof of Claim 7. In  section 6 we show that the first term of the left endpoint expansion of the Gr\"unwald formula approximation of the Caputo derivative ot the power function $t^\bbbb$ is also $\zzzz(-\bbbb) t^{-1-\aaaa}h^{1+\bbbb}/\GGGG(-\aaaa)$.
\section{Induced shifted approximations of the Caputo derivative}
  In this section we use the method from \cite{DimitrovMiryanovTodorov2018} to obtain the second order induced shifted approximation of the Caputo derivative of the Gr\"unwald formula and the induced shifted  approximations   of \eqref{2_1}, \eqref{4_2}, \eqref{4_3}, \eqref{5_1} of order $2-\aaaa$. At the optimal shift values the approximations have a third order  and a second order accuracy respectively. The construction of the induced shifted approximations is based on Lemma 9 and Lemma 10. Denote by
$\mathcal{G}_n^{(\aaaa)}[y(t)]$ the Gr\"unwald formula approximation
	\begin{equation}\label{Grunwald1}
	\mathcal{G}_n^{(\aaaa)}[y(t)]=\sum_{k=0}^n (-1)^k \binom{\aaaa}{k}=y_{n-\aaaa/2}^{(\aaaa)}+O(h^2).
	\end{equation}
	\begin{lem} Suppose that $y(0)=y'(0)=0$. Then
$$\mathcal{G}_n^{(\aaaa)}[y(t)]-(s-\aaaa/2)h\mathcal{G}_n^{(1+\aaaa)}[y(t)]=y_{n-s}^{(\aaaa)}+O(h^2).$$
\end{lem}
\begin{proof} The Gr\"unwald formula approximation of the Caputo derivative has a first order accuracy
$$\mathcal{G}_n^{(1+\aaaa)}y(t)=\dddd{1}{h^{1+\aaaa}}\sum_{k=0}^{n}\binom{1+\aaaa}{k}y_{n-k}=y_n^{(1+\aaaa)}+O(h).$$
The function $y(t)$ satisfies $y(0)=y'(0)=0$ and
\begin{align*}
(s-\aaaa/2)h\mathcal{G}_n^{(1+\aaaa)}[y(t)]=(s-\aaaa/2)h \dddd{d}{d t}y_n^{(\aaaa)}+O(h^2),
\end{align*}
$$\mathcal{G}_n^{(\aaaa)}[y(t)]-(s-\aaaa/2)h\mathcal{G}_n^{(1+\aaaa)}[y(t)]=y_{n}^{(\aaaa)}-s h \dddd{d}{d t}y_n^{(\aaaa)}+O(h^2)=y_{n-s}^{(\aaaa)}+O(h^2).$$
\end{proof}
\noindent
Let $\mathcal{A}_n^{(\aaaa)}[y(t)]$ be an approximation of the Caputo derivative of order $2-\aaaa$.
$$\mathcal{A}_n^{(\aaaa)}[y(t)]=\dddd{1}{h^\aaaa}\sum_{k=0}^n \llllll_k^{(\aaaa)}y_{n-k}=y_n^{(\aaaa)}+O\llll(h^{2-\aaaa}\rrrr).$$
 Denote
$$\mathcal{B}_n^{(\aaaa)}[y(t)]=\mathcal{A}_n^{(\aaaa)}[y(t)]-s h \mathcal{A}_n^{(1+\aaaa)}[y(t)]=\dddd{1}{h^\aaaa}\sum_{k=0}^n \llll(\llllll_k^{(\aaaa)}-s \llllll_k^{(1+\aaaa)}\rrrr)y_{n-k}.$$
\begin{lem} Let  $y(0)=y'(0)=0$ and 
$$\mathcal{A}_n^{(\aaaa)}[y(t)]=y_n^{(\aaaa)}+O(h^{2-\aaaa}),\quad \mathcal{A}_n^{(1+\aaaa)}[y(t)]=y_n^{(1+\aaaa)}+O(h^{1-\aaaa}).
$$
Then
$$\mathcal{B}_n^{(\aaaa)}[y(t)]=y_{n-s}^{(\aaaa)}+O(h^{2-\aaaa}).
$$
\end{lem}
\noindent
The proof of Lemma 10 is similar to the proof of Lemma 9.
\subsection{Shifted Gr\"unwald formula approximations}
Denote
$$\mathcal{H}_n^{(\aaaa)}[y(t)]=\mathcal{G}_n^{(\aaaa)}[y(t)]-(s-\aaaa/2)h\mathcal{G}_n^{(1+\aaaa)}[y(t)]$$
From Lemma 10  approximation $\mathcal{H}_n^{(\aaaa)}[y(t)]$
is  a second order shifted approximation of the Caputo derivative with a shift parameter $s$.
$$\mathcal{H}_n^{(\aaaa)}[y(t)]=\dddd{1}{h^\aaaa}\sum_{k=0}^{n}\llll(\binom{\aaaa}{k}-(s-\aaaa/2)\binom{1+\aaaa}{k}\rrrr)y_{n-k}.$$
The weights $g_k^{(\aaaa)}$ of approximation  $\mathcal{H}_n^{(\aaaa)}[y(t)]$ satisfy
\begin{align*}
g_k^{(\aaaa)}=\binom{\aaaa}{k}-(s-\aaaa/2)\binom{1+\aaaa}{k}=
&\binom{\aaaa+1}{k}\dddd{(\aaaa+1)(1+\aaaa/2-s)-k}{\aaaa+1}.
\end{align*}
Hence
\begin{align}\label{Gg}
\mathcal{H}_n^{(\aaaa)}[y(t)]=\dddd{1}{h^\aaaa}\sum_{k=0}^{n}g_k^{(\aaaa)}y_{n-k}=y_{n-s}^{(\aaaa)}+O(h^2),
\end{align}
where
$$g_k^{(\aaaa)}=(-1)^k \llll(\dddd{2+\aaaa-2s}{2}-\dddd{k}{1+\aaaa}\rrrr)\binom{\aaaa+1}{k}.$$
Shifted approximation \eqref{Gg} has a second order accuracy when  $y\in C^2[0,t_n]$ and satisfies $y(0)=y'(0)=0$.  When the shift parameter $s=\aaaa/2$ approximation \eqref{Gg} is the shifted Gr\"unwald formula  and when $s=0$,  approximation \eqref{Gg} is a second order approximation of the Caputo derivative \eqref{SOGLA}. 
In \cite{DimitrovMiryanovTodorov2018} we construct a modification of the Gr\"unwald formula approximation which has a second order accuracy for all functions in the class $C^2[0,t_n]$. Now we apply the method from \cite{DimitrovMiryanovTodorov2018} for modifying the last two weights of approximation \eqref{Gg}. The modified shifted approximation has a second order accuracy for all functions of the class $C^2[0,t_n]$ and satisfies 
$$\mathcal{H}_n^{(\aaaa)}[1]=0,\; \mathcal{H}_n^{(\aaaa)}[t]=(t-s)^{1-\aaaa}/\GGGG(2-\aaaa).$$ 
Let $y(t)=y_0+y'_0 t+z(t)$. The function $z(t)$ satisfies $z(0)=z'(0)=0$ and 
\begin{align*}
\mathcal{H}_n^{(\aaaa)}[z(t)]=z_{n-s}^{(\aaaa)}+O(h^2).
\end{align*}
The Caputo derivatives of the functions $y(t)$ and $z(t)$ are related as 
$$y^{(\aaaa)}(t-s h)=y'_0 (t-s h)^{1-\aaaa}/\GGGG(2-\aaaa)+z^{(\aaaa)}(t-s h).$$ 
Hence
$$\mathcal{H}_n^{(\aaaa)}[y(t)]=y_0 \mathcal{H}_n^{(\aaaa)}[1]+y'_0 \mathcal{H}_n^{(\aaaa)}[t]+\mathcal{H}_n^{(\aaaa)}[z(t)],$$
$$\mathcal{H}_n^{(\aaaa)}[y(t)]=y_0 \mathcal{H}_n^{(\aaaa)}[1]+y'_0 \llll(\mathcal{H}_n^{(\aaaa)}[t]-\dddd{(t-s h)^{1-\aaaa}}{\GGGG(2-\aaaa)}\rrrr)+y^{(\aaaa)}(t-s h)+O(h^2),$$
\begin{align}\label{H11}
\mathcal{H}_n^{(\aaaa)}[y(t)]=y^{(\aaaa)}(t-s)+\dddd{1}{h^\aaaa}\llll(\mathcal{W}_N^0y_0 +\mathcal{W}_N^1 y'_0 h\rrrr)+O(h^2),\end{align}
where
$$\mathcal{W}_N^0=h^\aaaa \mathcal{H}_n^{(\aaaa)}[1],\quad \mathcal{W}_N^1=h^{\aaaa-1}\llll( \mathcal{H}_n^{(\aaaa)}[t]-\dddd{(t-s h)^{1-\aaaa}}{\GGGG(2-\aaaa)} \rrrr).
$$
The number $\mathcal{W}_N^1$ satisfies the asymptotic estimate $\mathcal{W}_N^1= O(h^{1+\aaaa})$. The proof is similar to the proof of Lemma 2 in \cite{DimitrovMiryanovTodorov2018}. By approximating $y'_0$ in \eqref{H11} by a first order backward difference 
$h y'_0=y_1-y_0+O(h^2)$
we obtain the second order shifted approximation
$$\mathcal{H}_n^{(\aaaa)}[y(t)]=\dddd{1}{h^\aaaa}\sum_{k=0}^{N}\gggg_k^{(\aaaa)}y_{n-k}=y_{n-s}^{(\aaaa)}+O(h^2),$$
where 
$$\gggg_k^{(\aaaa)}=g_k^{(\aaaa)}=(-1)^k \binom{\aaaa}{k}\frac{\aaaa^2-2 \aaaa s +3 \aaaa-2 k-2 s +2}{2 (\aaaa-k+1)},\quad (0\leq k\leq n-2),$$
\begin{align}\label{LWeights}
\gggg_N^{(\aaaa)}=\mathcal{W}_N^1-\mathcal{W}_N^0,\quad \gggg_{N-1}^{(\aaaa)}=g_{N-1}^{(\aaaa)}-\mathcal{W}_N^1.
\end{align}
The binomial coefficients satisfy the identities \cite{Podlubny1999} 
\begin{align*}
\sum_{k=0}^{N-1}(-1)^k \binom{\aaaa}{k}=(-1)^{N-1}\binom{\aaaa-1}{N-1},
\end{align*}
\begin{align}\label{Identity2}
\sum_{k=0}^{N-1}(-1)^k k\binom{\aaaa}{k}=\aaaa(-1)^{N-1}\binom{\aaaa-2}{N-2}.
\end{align}
Hence
\begin{align*}
\mathcal{W}_N^0=\sum_{k=0}^{N-1}g_k^{(\aaaa)}&=\sum_{k=0}^{N-1}(-1)^k\binom{\aaaa}{k}-(s-\aaaa/2)\sum_{k=0}^{N-1}(-1)^k\binom{1+\aaaa}{k}\\
&=(-1)^{N-1}\llll(\binom{\aaaa-1}{N-1}-(s-\aaaa/2)\binom{\aaaa}{N-1}\rrrr),
\end{align*}
\begin{align}\label{WN0}
\mathcal{W}_N^0=(-1)^{N-1}\binom{\aaaa}{N-1}\dddd{\aaaa^2-2\aaaa s+2\aaaa+2-2N}{2\aaaa}.
\end{align}
Now we obtain a formula for $\mathcal{W}_N^1$. 
$$\mathcal{H}_N^{(\aaaa)}[t]=\sum_{k=0}^{N-1}g_k^{(\aaaa)}(t-k h)=t\sum_{k=0}^{N-1}g_k^{(\aaaa)}-h \sum_{k=0}^{N-1}k g_k^{(\aaaa)}.$$
From \eqref{WN0} and $t=Nh$:
$$t\sum_{k=0}^{N-1}g_k^{(\aaaa)}=h N \mathcal{W}_N^0=h(-1)^{N-1}\binom{\aaaa}{N-1}\dddd{N(\aaaa^2-2\aaaa s+2\aaaa+2-2N)}{2\aaaa}.
$$
From \eqref{Identity2}
\begin{align*}
\sum_{k=0}^{N-1}&k g_k^{(\aaaa)}=\sum_{k=0}^{N-1}k \binom{\aaaa}{k}-(s-\aaaa/2)\sum_{k=0}^{N-1}k \binom{1+\aaaa}{k}\\
&=\aaaa (-1)^{N-1}\binom{\aaaa-2}{N-2}-(s-\aaaa/2)(1+\aaaa) (-1)^{N-1}\binom{\aaaa-1}{N-2}\\
&= (-1)^{N-1}\binom{\aaaa}{N-1}\frac{(N-1)(\aaaa^3-2 \aaaa^2 s +2 \aaaa^2-2 \aaaa N+\aaaa+2 s) }{2\aaaa (\aaaa-1)}.
\end{align*}
Hence
\begin{align*}
\mathcal{W}_N^1=&(-1)^{N-1}\binom{\aaaa}{N-1}\dddd{N(\aaaa^2-2\aaaa s+2\aaaa+2-2N)}{2\aaaa}-\dddd{(N-s)^{1-\aaaa}}{\GGGG(2-\aaaa)}-\\
& (-1)^{N-1}\binom{\aaaa}{N-1}\frac{(N-1)(\aaaa^3-2 \aaaa^2 s +2 \aaaa^2-2 \aaaa N+\aaaa+2 s) }{2\aaaa (\aaaa-1)},
\end{align*}
\begin{align}\label{WN1}
\mathcal{W}_N^1=(-1)^{N-1}\binom{\aaaa}{N-1}W-\dddd{(N-s)^{1-\aaaa}}{\GGGG(2-\aaaa)},
\end{align}
where
$$W=\dddd{(\aaaa-N+1) \left(\aaaa^2-2 \aaaa s +\aaaa-2 N+2 s \right)}{2 \aaaa(\aaaa-1)}.
$$
From \eqref{LWeights}, \eqref{WN0}, \eqref{WN1} we obtain the the formulas for the weights $\gggg_{N-1}^{(\aaaa)}$ and  $\gggg_N^{(\aaaa)}$ and the second order approximation of the Caputo derivative
\begin{equation}\label{2ndSiftedGrunwald}
\dddd{1}{h^\aaaa}\sum_{k=0}^{N}\gggg_k^{(\aaaa)}y_{n-k}=y_{n-s}^{(\aaaa)}+O(h^2),
\end{equation}
where 
$$\gggg_k^{(\aaaa)}=(-1)^k \binom{\aaaa}{k}\frac{\aaaa^2-2 \aaaa s +3 \aaaa-2 k-2 s +2}{2 (\aaaa-k+1)},\quad (0\leq k\leq n-2),$$
$$\gggg_{n-1}^{(\aaaa)}=(-1)^{n-1}\binom{\aaaa-1}{n-2}\bar{W}+\dddd{(n-s)^{1-\aaaa}}{\GGGG(2-\aaaa)},
$$
$$\gggg_n^{(\aaaa)}=(-1)^n\binom{\aaaa-1}{n-2}\frac{\aaaa^2-2 \aaaa s +\aaaa-2 n+2 s +2}{2 (\aaaa-1)}-\dddd{(n-s)^{1-\aaaa}}{\GGGG(2-\aaaa)},
$$
where
$$\bar{W}=\frac{2 \aaaa^3-\aaaa^2 n-4 \aaaa^2 s +4 \aaaa^2+2 \aaaa n s -5 \aaaa n+6 \aaaa+2 n^2-2 n s -4 n+4 s }{2 (\aaaa-1) (\aaaa-n+2)}.
$$
Shifted approximation \eqref{2ndSiftedGrunwald} has a second order accuracy for all functions $y\in C^2[0,t_n]$. Now we obtain the optimal value of the shift parameter, where  approximation \eqref{2ndSiftedGrunwald} has a third order  accuracy.
The Gr\"unwald formula approximation $\mathcal{G}_n^{(\aaaa)}[y(t)]$ has a third order expansion 
$$\dddd{1}{h^\aaaa}\sum_{k=0}^{n}(-1)^k\binom{\aaaa}{k}y_{n-k}=y_n^{(\aaaa)}-\dddd{\aaaa}{2}h y_n^{(1+\aaaa)}+\dddd{\aaaa+3\aaaa^2}{24}h^2 y_n^{(2+\aaaa)}+O(h^3),$$
and $\mathcal{G}_n^{(1+\aaaa)}[y(t)]$ has a second order expansion formula
$$\mathcal{G}_n^{(1+\aaaa)}[y(t)]=\dddd{1}{h^{1+\aaaa}}\sum_{k=0}^{n}(-1)^k\binom{1+\aaaa}{k}y_{n-k}=y_n^{(1+\aaaa)}-\dddd{1+\aaaa}{2}h y_n^{(2+\aaaa)}+O(h^2).$$
Approximation \eqref{2ndSiftedGrunwald} satisfies
$$\mathcal{H}_n^{(\aaaa)}[y(t)]=y_n^{(\aaaa)}-s h y_n^{(1+\aaaa)}+\llll(\dddd{\aaaa+3\aaaa^2}{24}-(\aaaa/2-s)\dddd{1+\aaaa}{2}\rrrr)h^2 y_n^{(2+\aaaa)}+O(h^3),$$
$$\mathcal{H}_n^{(\aaaa)}[y(t)]=y_{n-s}^{(\aaaa)}+\llll(-\dddd{s^2}{2}+\dddd{\aaaa+3\aaaa^2}{24}-(\aaaa/2-s)\dddd{1+\aaaa}{2}\rrrr)h^2 y_n^{(2+\aaaa)}+O(h^3).$$
Shifted approximation \eqref{2ndSiftedGrunwald} has a third order accuracy when the coefficient of the second order term of the expansion formula is zero.
$$-\dddd{s^2}{2}+\dddd{\aaaa+3\aaaa^2}{24}-(\aaaa/2-s)\dddd{1+\aaaa}{2}=0,
$$
$$12 s ^2-12 s (\aaaa+1)+3 \aaaa^2 +5 \aaaa=0,\quad s = \frac{1}{6} \left(3 \aaaa+3\pm \sqrt{3(\aaaa+3)}\right).$$
The optimal shift value of  \eqref{2ndSiftedGrunwald} is
 $$s=S_1(\aaaa) = \dddd{1}{6}\left(3 \aaaa+3- \sqrt{3(\aaaa+3)}\right).$$
Shifted approximation \eqref{2ndSiftedGrunwald} has a third order accuracy when $s=S_1(\aaaa)$.
The first term of the left endpoint expansion of approximations \eqref{2_1}, \eqref{4_2}, \eqref{4_3} and \eqref{5_1} of the Caputo derivative of the power function is $\zeta (-\bbbb)t^{-\aaaa-1}h^{1+\bbbb}/\Gamma (-\aaaa)$. Now we show that  the left endpoint expansion of the shifted Gr\"unwald formula approximation has the same first term of order $1+\bbbb$. The weights of the Gr\"unwald formula approximation have asymptotic expansions  \cite{Elezovic2005,TricomiErdelyi1951}
\begin{align}\label{wk3a}
&w_k^{(\aaaa)}=(-1)^k\binom{\aaaa}{k}=\sum_{m=0}^M\dddd{B_m^{(-\aaaa)}(-\aaaa)}{m!\GGGG(-m-\aaaa)}\dddd{1}{k^{m+\aaaa+1}}+O\llll(\dddd{1}{k^{M+\aaaa+2}}\rrrr),\nonumber\\ 
&w_k^{(\aaaa)}=\dddd{1}{\GGGG(-\aaaa)k^{1+\aaaa}}-\dddd{\aaaa}{2\GGGG(-1-\aaaa)k^{2+\aaaa}}+O\llll(\dddd{1}{k^{3+\aaaa}}\rrrr).\end{align}
Let $y(t)=t^{\bbbb}$ and $0< \bbbb< 1$.  The first term of the left endpoint expansion of the Gr\"unwald formula is equal to the first term of the left endpoint expansion of the 
formula which has weights the first term of the expansion formula \eqref{wk3a} of  $w_k^{(\aaaa)}$:
\begin{align}\label{Formula12}
\dddd{1}{\GGGG(-\aaaa)h^\aaaa}\sum_{k=1}^{n}\dddd{(t-k h)^\bbbb}{k^{1+\aaaa}}.
\end{align}
 The proof is similar to the proof of Lemma 9. 
From expansion formula \eqref{AA1} with $\aaaa=\bbbb,\bbbb=-\aaaa-1$ the first term of the left endpoint expansion of \eqref{Formula12} is $ \zeta (-\bbbb)t^{-\aaaa-1}h^{\bbbb+1}/\Gamma (-\aaaa)$.
 The Gr\"unwald formula approximation of the power function has a second order expansion  when $0< \bbbb < 1$.
\begin{align}\label{2GFA}
\dddd{1}{h^\aaaa}\sum_{k=0}^{n}(-1)^k\binom{\aaaa}{k}(t-k h)^\bbbb=\dddd{\GGGG(\bbbb+1)}{\GGGG(\bbbb-\aaaa+1)}&\llll(t-\dddd{\aaaa h}{2}\rrrr)^{\bbbb-\aaaa}+\\
&\frac{ \zeta (-\bbbb)}{ \Gamma (-\aaaa)t^{\aaaa+1}}h^{\bbbb+1}+O(h^2).\nonumber\\
\end{align}
 By substituting $t=n h$ in \eqref{2GFA} we obtain the  expansion formula
\begin{align*}
\sum_{k=0}^{ n-1}(-1)^k\binom{\aaaa}{k}(n-k)^\bbbb=\dddd{\GGGG(\bbbb+1)}{\GGGG(\bbbb-\aaaa+1)}&\llll(n-\dddd{\aaaa }{2}\rrrr)^{\bbbb-\aaaa}+\\
&\frac{ \zeta (-\bbbb)}{ \Gamma (-\aaaa)n^{1+\aaaa}}+O\llll(\dddd{1}{n^{2+\aaaa-\bbbb}} \rrrr).
\end{align*}
\subsection{Numerical solutions of the two-term and three-term FDEs}
The analytical and the numerical solutions of the two-term and the three-term equations are studied in \cite{Diethelm2010,DiethelmSiegmundTuan2017,Dimitrov2015,Dimitrov2016,Dimitrov2017,DimitrovMiryanovTodorov2018,LiChenYe2011,Podlubny1999,ZengZhangKarniadakis2017}. 	
Let
\begin{equation}\tag{**}
\dddd{1}{h^\aaaa}\sum_{k=0}^{n}\llllll_k^{(\aaaa)} y_{n-k}= y_{n-s}^{(\aaaa)}+O\llll(h^{\bbbb(\aaaa)}\rrrr)
\end{equation}
be a shifted approximation for the Caputo derivative of order $\bbbb(\aaaa)\leq 3$. In the paper $\bbbb(\aaaa)=2-\aaaa,2,3$. 
	We derive the numerical solution of the two term equation of order $\min\{\bbbb(2\aaaa),3\}$ and the  numerical solution of the three-term equation of order $\min\{\bbbb(2\aaaa),2\}$,  which use shifted approximation (**) of the Caputo derivative.
\subsubsection{Two-term equation} 
\begin{align}\label{11_1}
y^{(\aaaa)}(t)+y(t)=F(t),\quad y(0)=y_0.
\end{align}
Let $h=1/N$, where $N$ is a positive integer.
 By approximating the Caputo derivative of equation \eqref{11_1} at the point $t_{n-s}=(n-s) h$ with (**)  we obtain
$$\dddd{1}{h^\aaaa}\sum_{k=0}^{n}\llllll_k^{(\aaaa)} y_{n-k}+ y_{n-s}= F_{n-s}+O\llll(h^{\bbbb(\aaaa)}\rrrr).$$
In \cite{Dimitrov2015} we showed that
$$\dddd{1}{2}s(s-1) y_{n-2}+s(2-s) y_{n-1}+\dddd{1}{2}(s-1)(s-2) y_{n}= y_{n-s}+O(h^3).$$
The  numerical solution $\{u_n\}_{n=0}^N$ of equation \eqref{11_1} is computed as
$$\dddd{1}{h^\aaaa}\sum_{k=0}^{n}\llllll_k^{(\aaaa)} u_{n-k}+\dddd{s(s-1)}{2} u_{n-2}+s(2-s) u_{n-1}+\dddd{(s-1)(s-2) }{2}u_{n}= F_{n-s},$$
\begin{align*}
\llll(\llllll_0^{(\aaaa)}+0.5 (s-1)(s-2) h^\aaaa\rrrr) u_{n}+&0.5 h^\aaaa s(s-1) u_{n-2}+\\
0.5 h^\aaaa s&(2-s) u_{n-1}+
\sum_{k=1}^{n}\llllll_k^{(\aaaa)} u_{n-k}= h^\aaaa F_{n-s},
\end{align*}
\begin{align}\tag{NS2(**)}
u_{n}=\dddd{h^\aaaa F_{n-s}-h^\aaaa \llll(0.5s(s-1) u_{n-2}+s(2-s) u_{n-1}\rrrr)-\sum_{k=1}^{n}\llllll_k^{(\aaaa)} u_{n-k}}{\llllll_0^{(\aaaa)}+0.5 (s-1)(s-2) h^\aaaa}.
\end{align}
	In \cite{Dimitrov2016} we showed that 
$$\bar{y}_1=\dddd{y(0)+\Gamma(2-\aaaa)h^\aaaa F(h)}{1+ \Gamma(2-\aaaa)h^\aaaa}$$
is a second order approximation for the value of the solution $y(h)$.
 Numerical solution NS2(**) of  equation \eqref{11_1} has second order initial conditions are $u_0=y_0, u_1=\bar{y}_1$. When the solution of the two-term equation satisfies $y(0)=y'(0)=y''(0)=0$, numerical solution NS2(**) has third order initial conditions $u_0=u_1=0$.
 The two-term equation
\begin{equation}\label{2e2}
y^{(\aaaa)}(t)+ y(t)=0,\quad y(0)=1
\end{equation}
has the solution $y(t)= E_{\aaaa}(-t^\aaaa)$, which has a singularity at the initial point $t=0$. The Gr\"unwald formula and approximations \eqref{2_1}, \eqref{4_2}, \eqref{4_3} and \eqref{5_1} of the Caputo derivative of the power function $t^\aaaa$ have an accuracy $O\llll(h^\aaaa\rrrr)$ for small $n$. The numerical solution of the two-term equation \eqref{2e2}  has an accuracy $O\llll(h^\aaaa\rrrr)$ \cite{DimitrovDimovTodorov2018, JinLazarovZhou2016}. Now we use the method from \cite{DimitrovDimovTodorov2018} for transforming equation \eqref{2e2} into a two-term equation which has a smooth solution. The Miller-Ross derivatives of the solution of equation \eqref{2e2}   satisfy:
$$y^{[n\aaaa]}(t)+y^{[(n-1)\aaaa]}(t)=0,$$
$$y^{[n\aaaa]}(0)=-y^{[(n-1)\aaaa]}(0)=(-1)^n.$$
 Substitute
$$z(t)=y(t)-T_m^{(\aaaa)}(t)=y(t)-\sum_{n=0}^m \dddd{\llll(-t^{\aaaa}\rrrr)^n}{\GGGG(\aaaa n+1)}.$$
The function $z(t)$ has a Caputo derivative of order $\aaaa$
$$z^{(\aaaa)}(t)=y^{(\aaaa)}(t)+\sum_{n=0}^{m-1} \dddd{\llll(- t^{\aaaa}\rrrr)^n}{\GGGG(\aaaa n+1)},$$
 and satisfies the two-term equation
\begin{equation}\label{TwoTerm2}
z^{(\aaaa)}(t)+B z(t)=\dddd{(-1)^{m+1} t^{\aaaa m}}{\GGGG(\aaaa m+1)},\quad z(0)=0.
\end{equation}
When $m\aaaa>3$ the solution of two-term equation \eqref{TwoTerm2} satisfies $z\in C^3[0,1]$ and $z(0)=z'(0)=z''(0)=0$. 
\subsubsection{Three-term equation}
\begin{equation}\label{TT1}
2y^{(2\aaaa)}(t)+3y^{(\aaaa)}(t)+y(t)=0,\quad y(0)=1,\quad (0<\aaaa<1/2).
\end{equation}
 In \cite{DimitrovDimovTodorov2018} we study the numerical solutions of the three-term equation which use approximation \eqref{2_4}  of the Caputo derivative. 
The numerical solution of three-term equation \eqref{TT1}, which uses the second order WSGL approximation is studied in \cite{ZengZhangKarniadakis2017}. The analytical solution of equation \eqref{TT1} has a fractional Taylor series expansion
$$y(t)=1+\dddd{y^{(\aaaa)}(0)t^\aaaa}{\GGGG(\aaaa+1)}+\sum_{n=2}^{\infty}\dddd{y^{[n\aaaa]}(0)t^{n\aaaa}}{\GGGG(\aaaa n+1)}.
$$
Compare the smallest power of $t$ in $y(t),\;y^{(\aaaa)}(t)$ and $y^{(2\aaaa)}(t)$. The smallest power is $t^{-\aaaa}$ of the term  $y^{(2\aaaa)}(0)t^{-\aaaa}/\GGGG(1-\aaaa)$, which is the Caputo derivative of order $2\aaaa$ of $y^{(\aaaa)}(0)t^\aaaa/\GGGG(\aaaa+1)$. Therefore the Caputo derivative of the solution of three-term equation \eqref{TT1} satisfies $y^{(\aaaa)}(0)=0$. The Caputo and Miller-Ross derivatives satisfy \cite{DimitrovDimovTodorov2018}:
$$y^{(2\aaaa)}(t)=y^{[2\aaaa]}(t)+\dddd{y^{(\aaaa)}(0)}{\GGGG(1-\aaaa)t^{\aaaa}}=y^{[2\aaaa]}(t).$$
Three-term equation \eqref{TT1} is formulated with the Miller-Ross fractional derivative as
\begin{equation}\label{TT2}
2y^{[2\aaaa]}(t)+3y^{[\aaaa]}(t)+y(t)=0,\quad y(0)=1, y^{(\aaaa)}(0)=0.
\end{equation}
Formulation \eqref{TT2} of three-term equation \eqref{TT1} has the advantage that $0<\aaaa<1$ as well as it has two independent initial conditions  \cite{DimitrovDimovTodorov2018}. Now we derive the analytical solution of three-term equation \eqref{TT2}. By applying fractional differentiation of order $\aaaa$ we obtain
$$2y^{[(n+1)\aaaa]}(t)+3y^{[n\aaaa]}(t)+y^{[(n-1)\aaaa]}(t)=0.$$
Denote $a_n=y^{[n\aaaa]}(0)$. The numbers $a_n$ satisfy
\begin{equation}\label{RR}
2a_{n+1}+3a_n+a_{n-1}=0,\quad a_0=1, a_1=0.
\end{equation}
Recurrence relations \eqref{RR} have a characteristic equation
$2r^2+3r+1=0,$ which has the solutions
 $r_1=-1/2,r_2=-1$. Hence
$$a_n=c_0\llll(-\dddd{1}{2}\rrrr)^n+c_1(-1)^n,$$
where the coefficients $c_0$ and $c_1$ satisfy the system of equations
		\begin{equation*}
	\left|
	\begin{array}{l l}
a_0=c_0+c_1=1,\\
a_1=c_0+2c_1=0. \\
	\end{array}
		\right . 
	\end{equation*}
Therefore $c_0=2,c_1=-1$ and $a_n=2\llll(-1/2\rrrr)^n-(-1)^n$. The solution of three-term equation \eqref{TT2} satisfies
	$$y(t)=1+\sum_{n=1}^{\infty}\dddd{a_n t^{n\aaaa}}{\GGGG(\aaaa n+1)}=2\sum_{n=0}^{\infty}\dddd{(-t^{\aaaa}/2)^n}{\GGGG(\aaaa n+1)}-\sum_{n=0}^{\infty}\dddd{\llll(-t^{\aaaa}\rrrr)^n}{\GGGG(\aaaa n+1)},$$
 $$y(t)=2E_\aaaa\llll(-t^{\aaaa}/2\rrrr)-E_\aaaa\llll(-t^{\aaaa}\rrrr).$$
  Now we transform equation \eqref{TT2} into a three-term FDE, which has a smooth solution. Substitute 
$$z(t)=y(t)-T^{(\aaaa)}_m(t)=y(t)-\sum_{n=0}^{m}\dddd{a_n t^{n\aaaa}}{\GGGG(\aaaa n+1)}$$
 The function $z$ satisfies $z(0)= z^{(\aaaa)}(0)=0$, when $m>1$ and
$$z^{(\aaaa)}(t)=y^{(\aaaa)}(t)-\sum_{n=0}^{m-1}\dddd{a_{n+1} t^{n\aaaa}}{\GGGG(\aaaa n+1)},\quad z^{[2\aaaa]}(t)=y^{[\aaaa]}(t)-\sum_{n=0}^{m-2}\dddd{a_{n+2} t^{n\aaaa}}{\GGGG(\aaaa n+1)}.$$
The  Caputo and Miller-Ross derivatives of the function $z(t)$ are equal  and the function $z(t)$ satisfies the three-term FDE
$$2z^{(2\aaaa)}(t)+3z^{(\aaaa)}(t)+z(t)=-\dddd{a_m t^{m\aaaa}}{\GGGG(\aaaa m+1)}-(3a_{m}+a_{m-1})\dddd{ t^{(m-1)\aaaa}}{\GGGG(\aaaa (m-1)+1)},$$
\begin{equation}\label{ThreeTerm}
2z^{(2\aaaa)}(t)+3z^{(\aaaa)}(t)+z(t)=F(t),z(0)=z^{(\aaaa)}(0)=0,
\end{equation}
where
$$F(t)=\dddd{ 2a_{m+1}t^{(m-1)\aaaa}}{\GGGG(\aaaa (m-1)+1)}-\dddd{a_m t^{m\aaaa}}{\GGGG(\aaaa m+1)}.$$
Now we obtain the numerical solution of three-term equation \eqref{ThreeTerm},  which uses approximation (**) of the Caputo derivative. By approximating the Caputo derivative at $t_{n-s}$ with (**) we obtain
\begin{align*}
\dddd{2}{h^{2\aaaa}}\sum_{k=0}^{n}\llllll_k^{(2\aaaa)} z_{n-k}+\dddd{3}{h^\aaaa}\sum_{k=0}^{n}\llllll_k^{(\aaaa)} z_{n-k}+z_{n-s}= F_{n-s}+O\llll(h^{\bbbb(2\aaaa)}\rrrr).
\end{align*}
The  numerical solution $\{u_n\}_{n=0}^N$ of three-term equation \eqref{ThreeTerm} satisfies
\begin{align*}
\dddd{2}{h^{2\aaaa}}\sum_{k=0}^{n}\llllll_k^{(2\aaaa)} u_{n-k}+\dddd{3}{h^\aaaa}\sum_{k=0}^{n}\llllll_k^{(\aaaa)} u_{n-k}+s u_{n-1}+(1-s)u_n= F_{n-s},
\end{align*}
\begin{align*}
2\sum_{k=0}^{n}\llllll_k^{(2\aaaa)} u_{n-k}+3 h^\aaaa\sum_{k=0}^{n}\llllll_k^{(\aaaa)} u_{n-k}+s h^{2\aaaa} u_{n-1}+(1-s)h^{2\aaaa}u_n= h^{2\aaaa} F_{n-s},
\end{align*}
\begin{align*}
\Big(2\llllll_0^{(2\aaaa)}+3 h^\aaaa\llllll_0^{(\aaaa)}+&h^{2\aaaa}(1-s) \Big)u_{n}=\\
& h^{2\aaaa}F_{n-s}-s h^{2\aaaa} u_{n-1}-\sum_{k=1}^{n}\llll(2\llllll_k^{(2\aaaa)}+3 h^\aaaa \llllll_k^{(\aaaa)}\rrrr) u_{n-k}.
\end{align*}
The numerical solution of  \eqref{ThreeTerm} is computed with
\begin{align}\tag{NS3(*)}
u_{n}=
\dddd{h^{2\aaaa}F_{n-s}-s h^{2\aaaa} u_{n-1}-\sum_{k=1}^{n}\llll(2\llllll_k^{(2\aaaa)}+3 h^\aaaa \llllll_k^{(\aaaa)}\rrrr) u_{n-k}}{2\llllll_0^{(2\aaaa)}+3 h^\aaaa\llllll_0^{(\aaaa)}+h^{2\aaaa}(1-s)}.
 \end{align}
 Numerical solution NS3(**) has third order initial conditions $u_0=u_1=0$, when the solution of three-term equation \eqref{ThreeTerm} satisfies $z(0)=z'(0)=z''(0)=0$.
 The numerical results for the error and the order of  numerical solution NS2\eqref{2ndSiftedGrunwald} of two-term equation \eqref{Equation2} with $\aaaa=0.3,s=0.25$, third order numerical solution NS2\eqref{2ndSiftedGrunwald} of two-term equation \eqref{TwoTerm2} with $\aaaa=0.6,m=5,s=S_1(0.6)=0.2523$ and  numerical solution NS3\eqref{2ndSiftedGrunwald} of three-term equation \eqref{ThreeTerm} with $\aaaa=0.8,m=4,s=0.3$ are presented in Table 8.
\subsection{Shifted approximations of order 2-\textalpha}
Approximations \eqref{2_1},\eqref{4_2}, \eqref{4_3}, \eqref{5_1} of the Caputo derivative satisfy the conditions of Lemma 10.  Now we use the method from Lemma 10 to construct their induced  shifted approximations. The induced shifted approximations  have  an order  $2-\aaaa$ and a second order accuracy at their optimal shift values. In  \cite{Dimitrov2018} we derive the expansion formula  of order $4-\aaaa$ of approximation \eqref{4_3}:
\begin{align}\label{EF43}
\mathcal{A}_n^{(\aaaa)}[y(&t)]=\dddd{1}{\GGGG(-\aaaa)h^\aaaa}\llll( \sum_{k=1}^{n-1}\dddd{y_{n-k}}{k^{1+\aaaa}}-\zzzz(1+\aaaa)y_n \rrrr)=y_n^{(\aaaa)}-\\
&\dddd{\zzzz(\aaaa)}{\GGGG(-\aaaa)}y'_n h^{1-\aaaa}+\dddd{\zzzz(\aaaa-1)}{2\GGGG(-\aaaa)}y''_n h^{2-\aaaa}-\dddd{\zzzz(\aaaa-2)}{6\GGGG(-\aaaa)}y'''_n h^{3-\aaaa}+O\llll(h^{4-\aaaa} \rrrr).\nonumber
\end{align}
By substituting $\aaaa:=\aaaa+1$ in \eqref{EF43} we obtain
\begin{align*}
&h\mathcal{A}_n^{(1+\aaaa)}[y(t)]=\dddd{1}{\GGGG(-\aaaa-1)h^\aaaa}\llll( \sum_{k=1}^{n-1}\dddd{y_{n-k}}{k^{2+\aaaa}}-\zzzz(2+\aaaa)y_n \rrrr)=h y_n^{(\aaaa+1)}-\\
&\dddd{\zzzz(1+\aaaa)}{\GGGG(-\aaaa-1)}y'_n h^{1-\aaaa}+\dddd{\zzzz(\aaaa)}{2\GGGG(-\aaaa-1)}y''_n h^{2-\aaaa}-\dddd{\zzzz(\aaaa-1)}{6\GGGG(-\aaaa-1)}y'''_n h^{3-\aaaa}+O\llll(h^{4-\aaaa} \rrrr).
\end{align*}
The gamma function satisfies $1/\GGGG(-\aaaa-1)=-(\aaaa+1)/\GGGG(-\aaaa)$. Hence
\begin{align*}
&-\dddd{\aaaa+1}{\GGGG(-\aaaa)h^\aaaa}\llll( \sum_{k=1}^{n-1}\dddd{y_{n-k}}{k^{2+\aaaa}}-\zzzz(2+\aaaa)y_n \rrrr)=h y_n^{(\aaaa+1)}-\dddd{(1+\aaaa)\zzzz(\aaaa)}{2\GGGG(-\aaaa)}y''_n h^{2-\aaaa}+\\
&\dddd{(1+\aaaa)\zzzz(1+\aaaa)}{\GGGG(-\aaaa)}y'_n h^{1-\aaaa}+\dddd{(1+\aaaa)\zzzz(\aaaa-1)}{6\GGGG(-\aaaa)}y'''_n h^{3-\aaaa}+O\llll(h^{4-\aaaa} \rrrr).
\end{align*}
Shifted approximation
$\mathcal{B}_n^{(\aaaa)}y(t)=\mathcal{A}_n^{(\aaaa)}y(t)-s h \mathcal{A}_n^{(1+\aaaa)}y(t)$ satisfies
\begin{align}\label{Bna}
\mathcal{B}_n^{(\aaaa)}[y(t)]&=\dddd{1}{\GGGG(-\aaaa)h^\aaaa}\Bigg( \sum_{k=1}^{n-1}w_k^{(\aaaa)}y_{n-k}-(\zzzz(1+\aaaa)+s(1+\aaaa)\zzzz(2+\aaaa))y_n \Bigg)=\nonumber\\
&=y_n^{(\aaaa)}-s h y_n^{(1+\aaaa)}-\dddd{\zzzz(\aaaa)+s(1+\aaaa)\zzzz(1+\aaaa)}{\GGGG(-\aaaa)}y'_n h^{1-\aaaa}+\\
&\qquad\qquad\qquad \dddd{\zzzz(\aaaa-1)+s(1+\aaaa)\zzzz(\aaaa)}{2\GGGG(-\aaaa)}y''_n h^{2-\aaaa}+O\llll(h^{3-\aaaa} \rrrr),\nonumber
\end{align}
where $w_k^{(\aaaa)}=1/k^{1+\aaaa}+s(1+\aaaa)/k^{2+\aaaa}$. From \eqref{Bna} and
$$ y^{(\aaaa)}_n-s h y^{(1+\aaaa)}_n=y^{(\aaaa)}_{n-s}-\dddd{s^2 h^2}{2}y^{(2+\aaaa)}_n+O(h^3)$$
we obtain
\begin{align}\label{Bna2}
\dddd{1}{\GGGG(-\aaaa)h^\aaaa}&\llll( \sum_{k=1}^{n-1}w_k^{(\aaaa)}y_{n-k}-(\zzzz(1+\aaaa)+s(1+\aaaa)\zzzz(2+\aaaa))y_n \rrrr)=\nonumber\\
&y^{(\aaaa)}_{n-s}-\dddd{\zzzz(\aaaa)+s(1+\aaaa)\zzzz(1+\aaaa)}{\GGGG(-\aaaa)}y'_n h^{1-\aaaa}+\\
&\dddd{\zzzz(\aaaa-1)+s(1+\aaaa)\zzzz(\aaaa)}{2\GGGG(-\aaaa)}y''_n h^{2-\aaaa}-\dddd{s^2 h^2}{2}y^{(2+\aaaa)}_n+O\llll(h^{3-\aaaa} \rrrr). \nonumber        
\end{align}
Substitute $h y'_n=y_n-y_{n-1}+h^2 y''_n/2+O(h^3)$ in \eqref{Bna2},
\begin{align}\label{Bna3}
&\dddd{1}{\GGGG(-\aaaa)h^\aaaa}\llll( \sum_{k=1}^{n-1}w_k^{(\aaaa)}y_{n-k}-(\zzzz(1+\aaaa)+s(1+\aaaa)\zzzz(2+\aaaa))y_n \rrrr)=\nonumber\\
&y^{(\aaaa)}_{n-s}-\dddd{\zzzz(\aaaa)+s(1+\aaaa)\zzzz(1+\aaaa)}{\GGGG(-\aaaa)}(y_n-y_{n-1}) h^{-\aaaa}-\dddd{s^2 h^2}{2}y^{(2+\aaaa)}_n+\\
&\dddd{\zzzz(\aaaa-1)+(s(1+\aaaa)-1)\zzzz(\aaaa)-s(1+\aaaa)\zzzz(1+\aaaa)}{2\GGGG(-\aaaa)}y''_n h^{2-\aaaa} +O\llll(h^{3-\aaaa} \rrrr).\nonumber        
\end{align}
Approximation \eqref{4_3} has an induced shifted approximation:
\begin{align}\label{Lap1}
\dddd{1}{\GGGG(-\aaaa)h^\aaaa} \sum_{k=0}^{n-1}w_k^{(\aaaa)}y_{n-k}=y^{(\aaaa)}_{n-s} +O\llll(h^{2-\aaaa} \rrrr),         
\end{align}
where
$$w_k^{(\aaaa)}=\dddd{1}{k^{1+\aaaa}}+\dddd{s(1+\aaaa)}{k^{2+\aaaa}},\qquad (k=2,\cdots,n),$$
$$w_0^{(\aaaa)}=\zzzz (\aaaa)+((\aaaa+1) s -1) \zzzz (\aaaa+1)-(\aaaa+1) s  \zzzz (\aaaa+2),$$
$$w_1^{(\aaaa)}=1+(\aaaa+1) s -\zzzz (\aaaa)-(\aaaa+1) s  \zzzz (\aaaa+1),$$
Shifted approximation \eqref{Lap1} has an order $2-\aaaa$ when the function $y(t)$ satisfies the condition $y(0)=y'(0)=0$.  Approximation \eqref{Lap1} has a second order accuracy when the coefficient of the term of order ${2-\aaaa}$ in expansion formula \eqref{Bna3} is equal to zero.
$$\zzzz(\aaaa-1)+(s(1+\aaaa)-1)\zzzz(\aaaa)-s(1+\aaaa)\zzzz(1+\aaaa)=0,$$
$$s=S_2(\aaaa)=\dddd{\zzzz(\aaaa)-\zzzz(\aaaa-1)}{(1+\aaaa)(\zzzz(\aaaa)-\zzzz(1+\aaaa))}.$$
The numerical results for the error and order of  numerical solution NS2\eqref{Lap1} of  two-term equation \eqref{TwoTerm2} with $\aaaa=0.3,m=7,s=0.25$ and $\aaaa=0.6,m=4,s=S_2(0.6)=0.2515$ and  numerical solution NS3\eqref{Lap1}  of three-term equation \eqref{ThreeTerm} with $\aaaa=0.4,m=6,s=0.2$ are presented in Table 9.

The L1 approximation has a second order expansion formula
\begin{align*}
\mathcal{A}_n^{(\aaaa)}[y(t)]=\dddd{1}{\GGGG(2-\aaaa)h^\aaaa}\sum_{k=1}^{n-1}w_k^{(\aaaa)}y_{n-k}=y_n^{(\aaaa)}+\dddd{\zzzz(\aaaa-1)}{\GGGG(2-\aaaa)}y''_n h^{2-\aaaa}+O\llll(h^{2} \rrrr),
\end{align*}
and $\mathcal{A}_n^{(1+\aaaa)}[y(t)]$ has an expansion of order $2-\aaaa$
\begin{align*}
\dddd{1}{\GGGG(1-\aaaa)h^{\aaaa+1}}\sum_{k=1}^{n-1}w_k^{(1+\aaaa)}y_{n-k}=y_n^{(1+\aaaa)}+
\dddd{\zzzz(\aaaa)}{\GGGG(1-\aaaa)}y''_n h^{1-\aaaa}+O\llll(h^{2-\aaaa} \rrrr).
\end{align*}
Approximation $\mathcal{B}_n^{(\aaaa)}[y(t)]=\mathcal{A}_n^{(\aaaa)}[y(t)]-s h \mathcal{A}_n^{(1+\aaaa)}[y(t)]$ satisfies
\begin{align*}
\mathcal{B}_n^{(\aaaa)}[y(t)]=\dddd{1}{\GGGG(2-\aaaa)h^\aaaa}\sum_{k=1}^{n-1}&\llll(w_k^{(\aaaa)}-s (1-\aaaa)w_k^{(1+\aaaa)}\rrrr)y_{n-k}=y_{n-s}^{(\aaaa)}+\\
&\dddd{\zzzz(\aaaa-1)-s (1-\aaaa)\zzzz(\aaaa)}{\GGGG(2-\aaaa)}y''_n h^{2-\aaaa}+O\llll(h^{2} \rrrr).
\end{align*}
The L1 approximation  has an induced shifted approximation:
\begin{align}\label{ShiftedL1}
\dddd{1}{\GGGG(2-\aaaa)h^\aaaa} \sum_{k=0}^{n-1}w_k^{(\aaaa)}y_{n-k}=y^{(\aaaa)}_{n-s} +O\llll(h^{2-\aaaa} \rrrr),         
\end{align}
where
$$w_0^{(\aaaa)}=1+s (\aaaa-1),w_1^{(\aaaa)}=2^{1-\aaaa}+s (\aaaa-1)2^{-\aaaa}-2(s (\aaaa-1)+1),$$
\begin{align*}
w_k^{(\aaaa)}=(k+1)&^{1-\aaaa}-2k^{1-\aaaa}+(k-1)^{1-\aaaa}+\\
&s (\aaaa-1)\llll( (k+1)^{-\aaaa}-2k^{-\aaaa}+(k-1)^{-\aaaa}  \rrrr),\quad (2\leq k \leq n-1),
\end{align*}
\begin{align*}
w_n^{(\aaaa)}=(n-1)^{1-\aaaa}-n^{1-\aaaa}+s (\aaaa-1)\llll( (n-1)^{-\aaaa}-n^{-\aaaa}  \rrrr).
\end{align*}
The optimal shift value of \eqref{ShiftedL1}, where the approximation has a second order accuracy is
$$s=S_3(\aaaa)=\dddd{\zzzz(\aaaa-1)}{(1-\aaaa)\zzzz(\aaaa)}.$$
The numerical results for the error and order of  numerical solution NS2\eqref{ShiftedL1}  of  two-term equation \eqref{TwoTerm2} with $\aaaa=0.3,m=7,s=0.25$ and  $\aaaa=0.6,m=4,s=S_3(0.6)=0.3164$ and  numerical solution NS3\eqref{ShiftedL1}  of three-term equation \eqref{ThreeTerm} with $\aaaa=0.4,m=5,s=0.2$ are presented in Table 10.
Using the method from Lemma 10, we obtain the induced shifted approximations of the Caputo derivative of approximations \eqref{4_2} and \eqref{5_1}.\\
\begin{align}\label{App1}
\dddd{1}{2\GGGG(1-\aaaa)h^\aaaa} \sum_{k=0}^{n-1}w_k^{(\aaaa)}y_{n-k}=y^{(\aaaa)}_{n-s} +O\llll(h^{2-\aaaa} \rrrr),       
\end{align}
where
$$w_0^{(\aaaa)}=1+\aaaa s -2 \aaaa s  \zzzz (\aaaa+1)-2 \zzzz (\aaaa),$$
$$w_1^{(\aaaa)}=\dddd{1}{2^{\aaaa}}+\dddd{\aaaa s}{2^{\aaaa+1}}  +2 \aaaa s  \zzzz (\aaaa+1)+2 \zzzz (\aaaa),$$
\begin{align*}
w_k^{(\aaaa)}=\dddd{1}{(k+1)^{\aaaa}}-\dddd{1}{(k-1)^{\aaaa}}+ \dddd{\aaaa s}{(k+1)^{\aaaa+1}}-\dddd{\aaaa s}{(k-1)^{\aaaa+1}},
\end{align*}
for $2\leq k\leq n$. Shifted approximation \eqref{App1} has an optimal shift value
$$s=S_4(\aaaa)=\dddd{2\zzzz(\aaaa-1)-\zzzz(\aaaa)}{\aaaa(\zzzz(\aaaa+1)-2\zzzz(\aaaa))}.$$
\begin{align}\label{App2}
\dddd{1}{\GGGG(1-\aaaa)h^\aaaa} \sum_{k=0}^{n-1}w_k^{(\aaaa)}y_{n-k}=y^{(\aaaa)}_{n-s} +O\llll(h^{2-\aaaa} \rrrr),     
\end{align}
where
$$w_0^{(\aaaa)}=2^\aaaa (2 \aaaa s +1)-\aaaa s \llll(2^{\aaaa+1}-1\rrrr)   \zzzz (\aaaa+1)-\llll(2^\aaaa-1\rrrr) \zzzz (\aaaa),$$
$$w_1^{(\aaaa)}=2^\aaaa \llll(2\aaaa s \llll(3^{-\aaaa-1}-1\rrrr)  +3^{-\aaaa}-1\rrrr)+\aaaa s \llll(2^{\aaaa+1}-1\rrrr)   \zzzz (\aaaa+1)+\llll(2^\aaaa-1\rrrr) \zzzz (\aaaa),$$
\begin{align*}
w_k^{(\aaaa)}=2^\aaaa \llll(\dddd{1}{(2 k+1)^{\aaaa}}-\dddd{1}{(2 k-1)^{\aaaa}}+ \dddd{2 \aaaa s }{(2 k+1)^{\aaaa+1}}-\dddd{2 \aaaa s }{(2 k-1)^{\aaaa+1}}\rrrr),
\end{align*}
for $2\leq k\leq n$.  Approximation \eqref{App2} has an optimal shift value
$$s=S_5(\aaaa)=\frac{\llll(2^\aaaa-1\rrrr) \zzzz (\aaaa)-\llll(2^\aaaa-2\rrrr) \zzzz (\aaaa-1)}{\aaaa (\llll(2^{\aaaa+1}-2\rrrr) \zzzz (\aaaa)-\llll(2^{\aaaa+1}-1\rrrr) \zzzz (\aaaa+1))}.$$
\begin{figure}[ht]
  \centering
  \caption{Graph of the optimal shift values $S_0(\aaaa),S_1(\aaaa)$, $S_2(\aaaa),S_3(\aaaa),S_4(\aaaa)$ and $S_5(\aaaa)$ of shifted approximations \eqref{Grunwald1},  \eqref{2ndSiftedGrunwald}, \eqref{Lap1}, \eqref{ShiftedL1}, \eqref{App1}, \eqref{App2}.} 
  \includegraphics[width=0.6\textwidth]{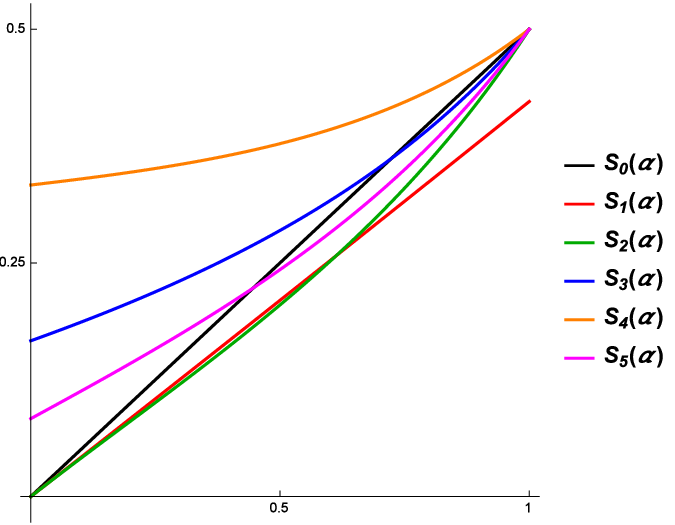}
\end{figure}\\
Shifted approximations  \eqref{Lap1}, \eqref{ShiftedL1}, \eqref{App1} and \eqref{App2} have an order $2-\aaaa$ when $y\in C^2[0,t_n]$ and $y(0)=y'(0)=0$, and second order accuracy at their optimal shift values.
The graphs of the optimal shift values  $S_0(\aaaa)$, $S_1(\aaaa)$, $S_2(\aaaa)$, $S_3(\aaaa)$, $S_4(\aaaa)$, $S_5(\aaaa)$ of  shifted approximations \eqref{Grunwald1},  \eqref{2ndSiftedGrunwald}, \eqref{Lap1}, \eqref{ShiftedL1}, \eqref{App1}, \eqref{App2} of the Caputo derivative, for $0<\aaaa<1$, are given in Figure 3. The numerical results for the  error and order of  numerical solutions NS2\eqref{App1} and NS2\eqref{App2}  of  two-term equation \eqref{TwoTerm2} with $\aaaa=0.3,m=7,s=0.25$, second order numerical solutions NS2\eqref{App1} and NS2\eqref{App1} of two-term equation \eqref{TwoTerm2} with $\aaaa=0.6,m=4,s=S_4(0.6)=0.3926$ and $\aaaa=0.5,m=5,s=S_5(0.5)=0.2428$ and  numerical solutions NS3\eqref{App1} and NS3\eqref{App2} of order $2-2\aaaa$ of three-term equation \eqref{ThreeTerm} with $\aaaa=0.4,m=5,s=0.2$. are presented in Table 11 and Table 12.
	\section{Conclusions}
 In the present paper we derive approximations of the Caputo derivative and their asymptotic expansions related to the midpoint approximation of the integral in the definition of the Caputo derivative. The L1 approximation and approximation \eqref{5_1} of the Caputo derivative have an order $2-\aaaa$ and properties \eqref{2_2} of the weights. According to the experimental results presented in the paper and the results from additional experiments with standard FDEs, approximation \eqref{5_1} has the advantage that the numerical solutions using \eqref{5_1} have smaller errors, which is explained with the smaller coefficient of the term of order ${2-\aaaa}$ in the expansion formula and the smaller truncation error of the approximation. In section 6 we derive the induced shifted approximations, and their optimal shift values, of the Grunwald formula and the approximations of the Caputo derivative studied in the paper.  A question for future work is to derive higher order shifted approximations of the Caputo derivative and to apply the approximations for numerical solution of ordinary and partial fractional differential equations.
\section{Experimental results}
The results of the numerical experiments  are presented in Tables 1-12.
	\setlength{\tabcolsep}{0.3em}
{ \renewcommand{\arraystretch}{1.1}
		\begin{table}[ht]
	\caption{ Maximum error and order of numerical solution NS1\eqref{2_1} of order $2-\aaaa$ of two-term equation \eqref{Equation1} and $\aaaa=0.25$, equation \eqref{Equation2} and $\aaaa=0.5$ and equation \eqref{Equation3} with $\aaaa=0.75$.}
	\small
	\centering
  \begin{tabular}{ |l | c  c | c  c | c  c| }
		\hline
		\multirow{2}*{ $\quad \boldsymbol{h}$}  & \multicolumn{2}{c|}{$\boldsymbol{\aaaa=0.25}$} & \multicolumn{2}{c|}{$\boldsymbol{\aaaa=0.5}$}  & \multicolumn{2}{c|}{$\boldsymbol{\aaaa=0.75}$} \\
		\cline{2-7}
   & $Error$ & $Order$  & $Error$ & $Order$  & $Error$ & $Order$ \\
		\hline
$0.0125$   & $0.4850\times 10^{-3}$ & $1.6660$ & $0.4009\times 10^{-3}$  & $1.4695$ & $0.1352\times 10^{-1}$ & $1.2280$ \\
$0.00625$  & $0.1510\times 10^{-3}$ & $1.6838$ & $0.1438\times 10^{-3}$  & $1.4792$ & $0.5729\times 10^{-2}$ & $1.2385$ \\
$0.003125$ & $0.4656\times 10^{-4}$ & $1.6970$ & $0.5134\times 10^{-4}$  & $1.4857$ & $0.2418\times 10^{-2}$ & $1.2442$ \\
\hline
  \end{tabular}
	\end{table}
	}
	\setlength{\tabcolsep}{0.3em}
{ \renewcommand{\arraystretch}{1.1}
		\begin{table}[ht]
	\caption{ Maximum error and order of numerical solution NS1\eqref{10_2} of order $1-\aaaa$ of two-term equation \eqref{Equation1} and $\aaaa=0.25$, equation \eqref{Equation2} and $\aaaa=0.5$ and equation \eqref{Equation3} with $\aaaa=0.75$.}
	\small
	\centering
  \begin{tabular}{ |l | c  c | c  c | c  c| }
		\hline
		\multirow{2}*{ $\quad \boldsymbol{h}$}  & \multicolumn{2}{c|}{$\boldsymbol{\aaaa=0.25}$} & \multicolumn{2}{c|}{$\boldsymbol{\aaaa=0.5}$}  & \multicolumn{2}{c|}{$\boldsymbol{\aaaa=0.75}$} \\
		\cline{2-7}
   & $Error$ & $Order$  & $Error$ & $Order$  & $Error$ & $Order$ \\
		\hline
$0.0125$   & $0.2132\times 10^{-1}$ & $0.7606$ & $0.5002\times 10^{-1}$  & $0.5136$ & $0.4984\times 10^{0}$ & $0.3052$ \\
$0.00625$  & $0.1262\times 10^{-1}$ & $0.7562$ & $0.3514\times 10^{-1}$  & $0.5096$ & $0.4061\times 10^{0}$ & $0.2955$ \\
$0.003125$ & $0.7486\times 10^{-2}$ & $0.7536$ & $0.2473\times 10^{-1}$  & $0.5067$ & $0.3324\times 10^{0}$ & $0.2887$ \\
\hline
  \end{tabular}
	\end{table}
	}
\setlength{\tabcolsep}{0.3em}
{ \renewcommand{\arraystretch}{1.1}
		\begin{table}[ht]
	\caption{ Maximum error and order of numerical solution NS1\eqref{13_2} of order $2-\aaaa$ of two-term equation \eqref{Equation1} and $\aaaa=0.25$, equation \eqref{Equation2} and $\aaaa=0.5$ and equation \eqref{Equation3} with $\aaaa=0.75$.}
	\small
	\centering
  \begin{tabular}{ |l | c  c | c  c | c  c| }
		\hline
		\multirow{2}*{ $\quad \boldsymbol{h}$}  & \multicolumn{2}{c|}{$\boldsymbol{\aaaa=0.25}$} & \multicolumn{2}{c|}{$\boldsymbol{\aaaa=0.5}$}  & \multicolumn{2}{c|}{$\boldsymbol{\aaaa=0.75}$} \\
		\cline{2-7}
   & $Error$ & $Order$  & $Error$ & $Order$  & $Error$ & $Order$ \\
		\hline
$0.0125$   & $0.3891\times 10^{-3}$ & $1.6940$ & $0.3568\times 10^{-3}$  & $1.4799$ & $0.1298\times 10^{-1}$ & $1.2334$ \\
$0.00625$  & $0.1192\times 10^{-3}$ & $1.7064$ & $0.1273\times 10^{-3}$  & $1.4868$ & $0.5489\times 10^{-2}$ & $1.2414$ \\
$0.003125$ & $0.3631\times 10^{-4}$ & $1.7152$ & $0.4529\times 10^{-4}$  & $1.4912$ & $0.2314\times 10^{-2}$ & $1.2460$ \\
\hline
  \end{tabular}
	\end{table}
	}

		\setlength{\tabcolsep}{0.3em}
{ \renewcommand{\arraystretch}{1.1}
		\begin{table}[ht]
	\caption{ Maximum error and order of second order numerical solution NS1\eqref{14_1}  of two-term equation \eqref{Equation1} and $\aaaa=0.25$, equation \eqref{Equation2} and $\aaaa=0.5$ and equation \eqref{Equation3} with $\aaaa=0.75$.}
	\small
	\centering
  \begin{tabular}{ |l | c  c | c  c | c  c| }
		\hline
		\multirow{2}*{ $\quad \boldsymbol{h}$}  & \multicolumn{2}{c|}{$\boldsymbol{\aaaa=0.25}$} & \multicolumn{2}{c|}{$\boldsymbol{\aaaa=0.5}$}  & \multicolumn{2}{c|}{$\boldsymbol{\aaaa=0.75}$} \\
		\cline{2-7}
   & $Error$ & $Order$  & $Error$ & $Order$  & $Error$ & $Order$ \\
		\hline
$0.0125$   & $0.6259\times 10^{-4}$ & $1.8774$ & $0.2387\times 10^{-4}$  & $1.9580$ & $0.1878\times 10^{-2}$ & $1.9220$ \\
$0.00625$  & $0.1642\times 10^{-4}$ & $1.9310$ & $0.6107\times 10^{-5}$  & $1.9660$ & $0.4825\times 10^{-3}$ & $1.9605$ \\
$0.003125$ & $0.4218\times 10^{-5}$ & $1.9602$ & $0.1554\times 10^{-5}$  & $1.9748$ & $0.1225\times 10^{-3}$ & $1.9781$ \\
\hline
  \end{tabular}
	\end{table}
	}
\setlength{\tabcolsep}{0.3em}
{ \renewcommand{\arraystretch}{1.1}
		\begin{table}[ht]
	\caption{ Maximum error and order of numerical solution \eqref{16_2} of the fractional subdiffusion equation \eqref{17_1} at time $t=1$ for $\aaaa=0.25,\aaaa=0.5,\aaaa=0.75$.}
	\small
	\centering
  \begin{tabular}{ |l | c  c | c  c | c  c| }
		\hline
		\multirow{2}*{ $\quad \boldsymbol{h}$}  & \multicolumn{2}{c|}{$\boldsymbol{\aaaa=0.25}$} & \multicolumn{2}{c|}{$\boldsymbol{\aaaa=0.5}$}  & \multicolumn{2}{c|}{$\boldsymbol{\aaaa=0.75}$} \\
		\cline{2-7}
   & $Error$ & $Order$  & $Error$ & $Order$  & $Error$ & $Order$ \\
		\hline
$0.0125$   & $0.2917\times 10^{-4}$ & $1.7773$ & $0.1514\times 10^{-3}$ & $1.5084$ & $0.7342\times 10^{-3}$  & $1.2499$ \\ 
$0.00625$  & $0.8528\times 10^{-5}$ & $1.7742$ & $0.5328\times 10^{-4}$ & $1.5067$ & $0.3086\times 10^{-3}$  & $1.2503$ \\ 
$0.003125$ & $0.2498\times 10^{-5}$ & $1.7712$ & $0.1877\times 10^{-4}$ & $1.5053$ & $0.1297\times 10^{-3}$  & $1.2504$ \\
\hline
  \end{tabular}
	\end{table}
	}
\setlength{\tabcolsep}{0.3em}
{ \renewcommand{\arraystretch}{1.1}
		\begin{table}[!h]
		\caption{ Maximum error and order of numerical solution \eqref{16_2} of the fractional subdiffusion equation \eqref{17_2}  at time $t=1$ for $\aaaa=0.25,\aaaa=0.5,\aaaa=0.75$.}
	\small
	\centering
  \begin{tabular}{ |l | c  c | c  c | c  c| }
		\hline
		\multirow{2}*{ $\quad \boldsymbol{h}$}  & \multicolumn{2}{c|}{$\boldsymbol{\aaaa=0.25}$} & \multicolumn{2}{c|}{$\boldsymbol{\aaaa=0.5}$}  & \multicolumn{2}{c|}{$\boldsymbol{\aaaa=0.75}$} \\
		\cline{2-7}
   & $Error$ & $Order$  & $Error$ & $Order$  & $Error$ & $Order$ \\
		\hline 
$0.0125$   & $0.4063\times 10^{-3}$  & $1.1216$  & $0.8783\times 10^{-3}$  & $1.0818$    & $0.1641\times 10^{-2}$  & $1.0661$  \\ 
$0.00625$  & $0.1942\times 10^{-3}$  & $1.0652$  & $0.4251\times 10^{-3}$  & $1.0471$    & $0.7947\times 10^{-3}$  & $1.0453$  \\ 
$0.003125$ & $0.9483\times 10^{-4}$  & $1.0339$  & $0.2085\times 10^{-3}$  & $1.0274$    & $0.3884\times 10^{-3}$  & $1.0327$   \\
\hline
  \end{tabular}
	\end{table}
	}
	\setlength{\tabcolsep}{0.3em}
{ \renewcommand{\arraystretch}{1.1}
		\begin{table}[ht]
	\caption{ Maximum error and order of numerical solution \eqref{16_2} of the fractional subdiffusion equation \eqref{19_1}  at time $t=1$ for $\aaaa=0.25$ and $m=8$, $\aaaa=0.5$ and $m=4$ and $\aaaa=0.75,m=2$.}
	\small
	\centering
  \begin{tabular}{ |l | c  c | c  c | c  c| }
		\hline
		\multirow{2}*{ $\quad \boldsymbol{h}$}  & \multicolumn{2}{c|}{$\boldsymbol{\aaaa=0.25,m=8}$} & \multicolumn{2}{c|}{$\boldsymbol{\aaaa=0.5,m=4}$}  & \multicolumn{2}{c|}{$\boldsymbol{\aaaa=0.75,m=2}$} \\
		\cline{2-7}
   & $Error$ & $Order$  & $Error$ & $Order$  & $Error$ & $Order$ \\
		\hline 
$0.0125$   & $0.2484\times 10^{-4}$  & $1.8641$  & $0.9292\times 10^{-4}$  & $1.5449$    & $0.4714\times 10^{-3}$ & $1.2690$ \\ 
$0.00625$  & $0.6842\times 10^{-5}$  & $1.8534$  & $0.3212\times 10^{-4}$  & $1.5328$    & $0.1966\times 10^{-3}$ & $1.2615$ \\ 
$0.003125$ & $0.1916\times 10^{-5}$  & $1.8431$  & $0.1117\times 10^{-4}$  & $1.5238$    & $0.8228\times 10^{-4}$ & $1.2570$  \\
\hline
  \end{tabular}
	\end{table}
	}
	\setlength{\tabcolsep}{0.3em}
{ \renewcommand{\arraystretch}{1.02}
\begin{table}[ht]
	\caption{ Maximum error and order of second order numerical solution NS2\eqref{2ndSiftedGrunwald} of  equation \eqref{Equation2} with $\aaaa=0.3,s=0.25$, third order numerical solution NS2\eqref{2ndSiftedGrunwald} of two-term equation \eqref{TwoTerm2} with $\aaaa=0.6,m=5,s=S_1(0.6)=0.2523$ and second order numerical solution NS3\eqref{2ndSiftedGrunwald} of three-term equation \eqref{ThreeTerm} with $\aaaa=0.8,m=4,s=0.3$.}
		\vspace{-0.2cm}
	\small
	\centering
  \begin{tabular}{ |l | c  c | c  c | c  c| }
		\hline
		\multirow{2}*{ $\quad \boldsymbol{h}$}  & \multicolumn{2}{c|}{$\boldsymbol{\aaaa=0.3,s=0.25}$} & \multicolumn{2}{c|}{$\boldsymbol{\aaaa=0.6,s=0.252}$}  & \multicolumn{2}{c|}{$\boldsymbol{\aaaa=0.8,s=0.3}$} \\
		\cline{2-7}
   & $Error$ & $Order$  & $Error$ & $Order$  & $Error$ & $Order$ \\
		\hline
$0.0125$   & $0.1260\times 10^{-4}$ & $2.0023$ & $0.3703\times 10^{-7}$  & $2.9815$ & $0.3234\times 10^{-4}$ & $1.9431$ \\
$0.00625$  & $0.3149\times 10^{-5}$ & $2.0005$ & $0.4656\times 10^{-8}$  & $2.9916$ & $0.8231\times 10^{-5}$ & $1.9743$ \\
$0.003125$ & $0.7875\times 10^{-6}$ & $1.9998$ & $0.5835\times 10^{-9}$  & $2.9961$ & $0.2075\times 10^{-5}$ & $1.9880$ \\
\hline
  \end{tabular}
	\end{table}	}
	\vspace{-0.5cm}
		\setlength{\tabcolsep}{0.3em}
{ \renewcommand{\arraystretch}{1.02}
	\begin{table}[ht]
	\caption{ Maximum error and order of  numerical solution NS2\eqref{Lap1} of order $2-\aaaa$ of  two-term equation \eqref{TwoTerm2} with $\aaaa=0.3,m=7,s=0.25$, second order numerical solution NS2\eqref{Lap1} of two-term equation \eqref{TwoTerm2} with $\aaaa=0.6,m=4,s=S_2(0.6)=0.2515$ and  numerical solution NS3\eqref{Lap1} of order $2-2\aaaa$ of three-term equation \eqref{ThreeTerm} with $\aaaa=0.4,m=6,s=0.2$.}
		\vspace{-0.2cm}
	\small
	\centering
  \begin{tabular}{ |l | c  c | c  c | c  c| }
		\hline
		\multirow{2}*{ $\quad \boldsymbol{h}$}  & \multicolumn{2}{c|}{$\boldsymbol{\aaaa=0.3,s=0.25}$} & \multicolumn{2}{c|}{$\boldsymbol{\aaaa=0.6,s=0.251}$}  & \multicolumn{2}{c|}{$\boldsymbol{\aaaa=0.4,s=0.2}$} \\
		\cline{2-7}
   & $Error$ & $Order$  & $Error$ & $Order$  & $Error$ & $Order$ \\
		\hline
$0.0125$   & $0.1444\times 10^{-4}$ & $1.6669$ & $0.1190\times 10^{-5}$  & $1.8400$ & $0.2769\times 10^{-3}$ & $1.0590$ \\
$0.00625$  & $0.4523\times 10^{-5}$ & $1.6747$ & $0.3197\times 10^{-6}$  & $1.8959$ & $0.1252\times 10^{-3}$ & $1.1458$ \\
$0.003125$ & $0.1411\times 10^{-5}$ & $1.6801$ & $0.8394\times 10^{-7}$  & $1.9292$ & $0.5552\times 10^{-4}$ & $1.1727$ \\
\hline
  \end{tabular}
	\end{table}	}
	\vspace{-0.5cm}
		\setlength{\tabcolsep}{0.3em}
{ \renewcommand{\arraystretch}{1.02}
		\begin{table}[ht]
	\caption{ Maximum error and order of  numerical solution NS2\eqref{ShiftedL1} of order $2-\aaaa$ of  two-term equation \eqref{TwoTerm2} with $\aaaa=0.3,m=7,s=0.25$, second order numerical solution NS2\eqref{ShiftedL1} of two-term equation \eqref{TwoTerm2} with $\aaaa=0.6,m=4,s=S_3(0.6)=0.3164$ and  numerical solution NS3\eqref{ShiftedL1} of order $2-2\aaaa$ of three-term equation \eqref{ThreeTerm} with $\aaaa=0.4,m=5,s=0.2$.}
		\vspace{-0.2cm}
	\small
	\centering
  \begin{tabular}{ |l | c  c | c  c | c  c| }
		\hline
		\multirow{2}*{ $\quad \boldsymbol{h}$}  & \multicolumn{2}{c|}{$\boldsymbol{\aaaa=0.3,s=0.25}$} & \multicolumn{2}{c|}{$\boldsymbol{\aaaa=0.6,s=0.316}$}  & \multicolumn{2}{c|}{$\boldsymbol{\aaaa=0.4,s=0.2}$} \\
		\cline{2-7}
   & $Error$ & $Order$  & $Error$ & $Order$  & $Error$ & $Order$ \\
		\hline
$0.0125$   & $0.4408\times 10^{-5}$ & $1.8494$ & $0.1652\times 10^{-5}$  & $2.0178$ & $0.3622\times 10^{-3}$  & $1.1009$ \\
$0.00625$  & $0.1233\times 10^{-5}$ & $1.8379$ & $0.4093\times 10^{-6}$  & $2.0126$ & $0.1608\times 10^{-3}$  & $1.1716$ \\
$0.003125$ & $0.3482\times 10^{-6}$ & $1.8244$ & $0.1017\times 10^{-6}$  & $2.0083$ & $0.7030\times 10^{-4}$  & $1.1938$ \\
\hline
  \end{tabular}
	\end{table}	}
		\setlength{\tabcolsep}{0.3em}
{ \renewcommand{\arraystretch}{1.02}
			\begin{table}[ht]
	\caption{ Maximum error and order of  numerical solution NS2\eqref{App1} of order $2-\aaaa$ of  two-term equation \eqref{TwoTerm2} with $\aaaa=0.3,m=7,s=0.25$, second order numerical solution NS2\eqref{App1} of two-term equation \eqref{TwoTerm2} with $\aaaa=0.6,m=4,s=S_4(0.6)=0.3926$ and  numerical solution NS3\eqref{App1} of order $2-2\aaaa$ of three-term equation \eqref{ThreeTerm} with $\aaaa=0.4,m=5,s=0.2$.}
		\vspace{-0.2cm}
	\small
	\centering
  \begin{tabular}{ |l | c  c | c  c | c  c| }
		\hline
		\multirow{2}*{ $\quad \boldsymbol{h}$}  & \multicolumn{2}{c|}{$\boldsymbol{\aaaa=0.3,s=0.25}$} & \multicolumn{2}{c|}{$\boldsymbol{\aaaa=0.6,s=0.393}$}  & \multicolumn{2}{c|}{$\boldsymbol{\aaaa=0.4,s=0.2}$} \\
		\cline{2-7}
   & $Error$ & $Order$  & $Error$ & $Order$  & $Error$ & $Order$ \\
		\hline
$0.0125$   & $0.6032\times 10^{-5}$ & $1.3520$ & $0.4187\times 10^{-5}$  & $1.9825$ & $0.4543\times 10^{-3}$ & $1.1289$ \\
$0.00625$  & $0.2183\times 10^{-5}$ & $1.4663$ & $0.1055\times 10^{-5}$  & $1.9881$ & $0.1993\times 10^{-3}$ & $1.1892$ \\
$0.003125$ & $0.7540\times 10^{-6}$ & $1.5337$ & $0.2655\times 10^{-6}$  & $1.9912$ & $0.8623\times 10^{-4}$ & $1.2084$ \\
\hline
  \end{tabular}
	\end{table}	}
	\vspace{-0.5cm}
		\setlength{\tabcolsep}{0.3em}
{ \renewcommand{\arraystretch}{1.02}
	\begin{table}[h!]
	\caption{ Maximum error and order of  numerical solution NS2\eqref{App2} of order $2-\aaaa$ of  two-term equation \eqref{TwoTerm2} with $\aaaa=0.3,m=7,s=0.25$, second order numerical solution NS2\eqref{App2} of two-term equation \eqref{TwoTerm2} with $\aaaa=0.5,m=5,s=S_5(0.5)=0.2428$ and  numerical solution NS3\eqref{App2} of order $2-2\aaaa$ of three-term equation \eqref{ThreeTerm} and $\aaaa=0.4,m=6,s=0.2$.}
		\vspace{-0.2cm}
	\small
	\centering
  \begin{tabular}{ |l | c  c | c  c | c  c| }
		\hline
		\multirow{2}*{ $\quad \boldsymbol{h}$}  & \multicolumn{2}{c|}{$\boldsymbol{\aaaa=0.3,s=0.25}$} & \multicolumn{2}{c|}{$\boldsymbol{\aaaa=0.5,s=0.243}$}  & \multicolumn{2}{c|}{$\boldsymbol{\aaaa=0.4,s=0.2}$} \\
		\cline{2-7}
   & $Error$ & $Order$  & $Error$ & $Order$  & $Error$ & $Order$ \\
		\hline
$0.0125$   & $0.9541\times 10^{-5}$ & $1.1714$ & $0.6344\times 10^{-6}$  & $2.0819$ & $0.3177\times 10^{-3}$ & $1.0815$ \\
$0.00625$  & $0.2914\times 10^{-5}$ & $1.7111$ & $0.1528\times 10^{-6}$  & $2.0536$ & $0.1422\times 10^{-3}$ & $1.1597$ \\
$0.003125$ & $0.8909\times 10^{-6}$ & $1.7097$ & $0.3730\times 10^{-7}$  & $2.0344$ & $0.6258\times 10^{-4}$ & $1.1841$ \\
\hline
  \end{tabular}
	\end{table}}

\clearpage
\section*{Acknowledgements}
The third author is supported by the Bulgarian Academy of Sciences through the Program for Career Development of Young Scientists, Grant DFNP-17-88/28.07.2017, Project “Efficient Numerical Methods with an Improved Rate of Convergence for Applied Computational Problems”, by the Bulgarian National Science Fund under Project DN 12/5-2017, Project “Efficient Stochastic Methods and Algorithms for Large-Scale Problems”, by the Bulgarian National Science Fund under Project DN 12/4-2017, Project “Advanced Analytical and Numerical Methods for Nonlinear Differential Equations with Applications in Finance and Environmental Pollution” and by the Bulgarian National Science Fund under Bilateral Project DNTS/Russia 02/12-2018 "Development and investigation of finite-difference schemes of higher order of accuracy for solving applied problems of fluid and gas mechanics, and ecology".
 
 \bigskip \smallskip

 \it

 \noindent
$^1$ Department of Mathematics and Physics \\
University of Forestry \\
10 Sveti Kliment Ohridski Blv., \\
Sofia -- 1756, BULGARIA  \\[4pt]
  e-mail: yuri.dimitrov@ltu.bg
	\hfill Received: September 1, 2018 \\[12pt]
	$^2$ Institute of Mathematics and Informatics,\\
 Bulgarian Academy of Sciences, \\
 Department of Information Modeling, Acad. Georgi Bonchev Str., Block 8,\\
Sofia -- 1113, BULGARIA\\[4pt]
vtodorov@math.bas.bg\\[12pt]
$^2$ Institute of Information and Communication Technologies,\\ Bulgarian Academy of Sciences,\\ 
Department of Parallel Algorithms,  Acad. Georgi Bonchev Str., Block 25A,\\
Sofia -- 1113, BULGARIA\\[4pt]
venelin@parallel.bas.bg\\[12pt]
$^3$ Department of Statistics and Applied Mathematics \\
University of Economics\\
77 Knyaz Boris I Blvd.,\\
Varna -- 9002, BULGARIA  \\[4pt]
 e-mail: miryanov@ue-varna.bg

\end{document}